\documentclass[a4paper]{article}
\usepackage{amsthm,amsfonts,amsmath,amssymb,units}
\usepackage[abbrev,nobysame]{amsrefs}
\usepackage[cp1251]{inputenc}
\usepackage[english]{babel}
\usepackage[final]{graphicx}
\usepackage{setspace}
\usepackage[12pt]{extsizes}
\begin{document}
\newcommand{\per}{{\rm per}}
\newcommand{\res}{{\rm res}}
\newcommand{\supp}{{\rm supp}}
\newtheorem{theorem}{Theorem}
\newtheorem{lemma}{Lemma}
\newtheorem{utv}{Proposition}
\newtheorem{svoistvo}{Property}
\newtheorem{sled}{Corollary}
\newtheorem{con}{Conjecture}
\newtheorem{zam}{Remark}
\newtheorem{quest}{Question}
\newtheorem{claim}{Claim}

\author{V. N. Potapov\footnote{Faculty of Engineering and Natural Sciences, Sabanc{\i} University,
Istanbul, T\"{u}rkiye; \texttt{quasigroup349@gmail.com}}, A. A. Taranenko\footnote{Sobolev Institute of Mathematics, Novosibirsk, Russia; \texttt{taa@math.nsc.ru}}}
\title{Multidimensional convolution matrices and  perfect colorings of subspace hypergraphs applied for bent functions and related designs}
\date{October 14, 2026}

\maketitle

\begin{abstract}
The main aim of the present paper is to introduce new methods for the study of combinatorial designs related to bent functions. They are based on interpretations of convolution on finite abelian groups as multiplication by a multidimensional matrix and  designs as perfect colorings of subspace hypergraphs of $\mathbb{F}_2^n$. We establish a correspondence between eigenfunctions of convolution matrices and perfect colorings of subspace hypergraphs,  show that perfect colorings of subspace hypergraphs admit a characterization in terms of convolution and that two-valued eigenfunctions of subspace hypergraphs correspond to perfect colorings. As applications, we represent  partial difference sets, bent and plateaued Boolean functions, spreads, and strong bent partitions of  $\mathbb{F}_2^n$ as eigenfunctions of convolution matrices and as perfect colorings of subspace hypergraphs. We also find some eigenvalues of convolution matrices over $\mathbb{F}_2^n$  and $\mathbb{F}_3^n$.

 
\end{abstract}

\section{Introduction}

Analytical and algebraic approaches are widely applied for the study of combinatorial designs. The analytical approach deals with designs via convolution or the Fourier transform of their indicator functions. Carlet's monograph~\cite{carle.boolforcrypt} on Boolean functions can be considered as an example of this methodology.
  
The graph-theoretical approach interprets designs as subsets of vertices in highly symmetrical graphs that have extremal or regularity properties. In the most common setting, designs correspond to perfect colorings (equitable partitions) of graphs. The concept of equitable partitions was introduced by Delsarte in \cite{dels.assch}. Perfect colorings of graphs play a significant role in coding theory and the theory of combinatorial designs. Applications of perfect colorings of several graph families for coding theory are surveyed in papers~\cite{our.perfcolorHam, BorRifaZin.compregcodes} and in Chapter~1 of book~\cite{ShiSole.CRCinDRG}.
  
The linear-algebraic approach employs tools from linear algebra and spectral theory of ($2$-dimensional) matrices to derive necessary conditions for the existence of designs. This approach is largely enabled by the property that the spectrum the adjacency matrix of a graph include the spectrum of the parameter (quotient) matrix of any its perfect coloring. Classical results in this direction are Lloyd's theorem for perfect codes~\cite{lloyd.perfcode}, the Fisher and Bruck--Ryser--Chowla inequalities for block designs, and the Bierbrauer--Friedman and Rao inequalities for orthogonal arrays.

The novelty of this paper is that we treat convolution as multiplication by a $3$-dimensional matrix. This point of view is well known in computer science and applied mathematics (see, for example,~\cite{Hack.tesorconv, PanVenLeib.lintomulti, RakhOsel.fastconv}), but appears to be uninvestigated in modern algebraic hypergraph theory and the theory of combinatorial designs. One of the main objectives of our work is to fill this gap. We also show that this correspondence allows some designs to be linked to eigenfunctions of the convolution matrix and to be represented as perfect colorings of a highly symmetric (subspace) hypergraph.

Several notions of eigenvalues and eigenfunctions for hypergraphs have been proposed in the literature. In this paper, we adopt the definition based on eigenvalues and eigenfunctions of the multidimensional adjacency matrix (tensor) of a hypergraph. This approach was introduced by~Lim \cite{lim.eigentensor} and further developed for symmetric tensors by Qi~\cite{qi.eigentensor}. It relies on the theory of determinants of multidimensional matrices and resultants of multilinear systems developed in book~\cite{GeKaZe.multdet} by Gelfand, Kapranov, and Zelevinsky. In~\cite{CoopDut.hyperspec} Cooper and Dutle  computed eigenvalues of certain classes of hypergraphs and  expressed some properties of hypergraphs in terms of their eigenvalues. See also~\cite{ShetBat.spechypersurv} for a survey of recent results on spectral properties of hypergraphs.  

The concept and properties of perfect colorings of hypergraphs were studied in depth by Taranenko \cite{my.hyperperfcolor}. To the best of our knowledge, paper~\cite{PotAv.hypercolor} (see also preprint ArXiv:2403.02904 for a corrected version) is the only work that directly investigates the relationship between designs and perfect colorings of hypergraphs. In that paper, it was initiated the representation of bent functions and partial difference sets as perfect colorings. In particular, it was shown that transversals in hypergraphs  are perfect colorings. The present study continues and extends this line of research by establishing connections with convolution. This opens up the possibility of deriving new conditions for the existence of designs using the spectral theory of multidimensional matrices, and laying the groundwork for applying algebraic hypergraph theory to the study of combinatorial designs.

The paper is organised as follows. In Section~\ref{premsec}, we give the main definitions and necessary preliminary results. In particular, we define multidimensional convolution matrices and the subspace hypergraphs $\mathcal{H}_n$ and $\mathcal{D}_n$ for the space $\mathbb{F}_2^n$. We also recall some properties of perfect colorings of hypergraphs, eigenvalues of multidimensional matrices, and the Fourier transform.

The main part of the paper starts in Section~\ref{eigensec}, where we find some eigenvalues of the convolution matrices for additive groups of the spaces $\mathbb{F}_2^n$ and $\mathbb{F}_3^n$.  The spectral properties of the convolution matrix  of the  additive group of the space $\mathbb{F}_2^n$  are extremely important due to their relation to combinatorial designs: if the spectrum of the convolution matrix does not contain an eigenvalue  corresponding to some design, then such designs do not exist. We suppose that a complete description of the eigenspectrum of convolution matrices for $\mathbb{F}_2^n$  is an important open problem.

In Section~\ref{relationsec}, we establish general connections between convolution matrices, subspace hypergraphs, and perfect colorings. In particular, we show that  multiplication by the convolution matrix and multiplication by the adjacency matrix of a hypergraph are closely related. This property allows us to express some eigenvectors of one matrix by means of the other. Next, we prove that the definition of a perfect coloring of a subspace hypergraph $\mathcal{H}_n$ can be given in terms of convolution. At the end of the section, we introduce totally regular hypergraphs, generalizing subspace graphs, and show that every two-valued eigenfunction of totally regular hypergraphs is a perfect coloring.

Finally, Section~\ref{designsec} is devoted to applications of these methods to several designs closely related to bent functions. First, we prove that every partial difference set gives rise to an eigenfunction of the convolution matrix and that it is in one-to-one correspondence with a perfect coloring of a subspace hypergraph with certain parameters.  Next,  we show that plateaued and bent functions correspond to eigenfunctions of convolution matrices and to perfect colorings of subspace hypergraphs.  We conclude the section by proving that every $n/2$-spread in $\mathbb{F}_2^n$ is equivalent to a perfect coloring of a subspace hypergraph with certain parameters, and that every strong bent partition of $\mathbb{F}_2^n$   is a perfect coloring.

\section{Definitions and preliminaries} \label{premsec}

 A \textit{hypergraph} $\mathcal{G}$ is a pair of sets $(X,E)$, where $X$ is called the set of \textit{vertices} and $E$ is the set of \textit{hyperedges}; each $e \in E$ is a (nonempty) subset of $X$.  A vertex $x \in X$ is said to be \textit{incident} to a hyperedge $e$ if  $x \in e$.  
 
 The \textit{degree} of a vertex $x \in X$ is the number of hyperedges $e$ containing $x$. If all vertices of $\mathcal{G}$ have  degree $r$, then the hypergraph $\mathcal{G}$ is said to be \textit{$r$-regular}.     A hypergraph $\mathcal{G}$ is \textit{$d$-uniform}  if every hyperedge of $\mathcal{G}$ consists of $d$ vertices. Simple graphs are exactly $2$-uniform hypergraphs. 

In what follows, we use the notation $\mathbb{F}_2^n$  for an $n$-dimensional vector space over the prime field $\mathbb{F}_2$, as well as for the additive group of this space. The neutral element of $\mathbb{F}_2^n$ with respect to addition is $\overline{0} = (0, \ldots, 0)$. A subset $L \subseteq \mathbb{F}_2^n$  is a   \textit{linear subspace} if for all $x, y \in L$ we have $x+ y \in L$, and a subset $U \subseteq \mathbb{F}_2^n$ is an \textit{affine subspace} if there exists a linear subspace $L$ and $x \in \mathbb{F}_2^n$ such that $U = x + L$. 

Let the $3$-uniform \textit{subspace hypergraph $\mathcal{H}_n$} be the hypergraph with vertex set $\mathbb{F}^n_2 \setminus\{\overline{0}\}$ and hyperedges $\{ x,y, z\}$ such that $x+y+z=\overline{0}$.   We also define the $4$-uniform \textit{subspace hypergraph $\mathcal{D}_n$}  with the vertex set  $ \mathbb{F}^n_2 $  and  hyperedges  $\{ x,y,z,v \}$  such that $x + y+z+v = \overline{0}$.   Note that the hyperedges of $\mathcal{H}_n$ correspond to $2$-dimensional linear subspaces of the vector space $ \mathbb{F}^n_2 $ without $\overline{0}$, while  the  hyperedges of $\mathcal{D}_n$ correspond to $2$-dimensional affine subspaces of $\mathbb{F}^n_2$.

A \textit{$d$-dimensional matrix $A$ of order $n$} is an array $(a_{\textbf{x}})_{\textbf{x} \in X^d}$  whose entries $a_{\textbf{x}} \in\mathbb R$  are indexed by $d$-tuples $\textbf{x}  = (x_1, \ldots, x_d )$ from the index set $X^d$, where each $x_i$ belongs to a set $X$ of cardinality $n$.  $2$-dimensional matrices are standard matrices used in linear algebra, $1$-dimensional matrices are vectors. We often regard $1$-dimensional matrices (vectors) as functions $f : X \rightarrow \mathbb{R}$ and use notation $f(x)$ instead of $f_x$.

The \textit{adjacency matrix} $M(\mathcal{G})$ of a $d$-uniform hypergraph $\mathcal{G}  (X,E)$ on $n$ vertices is a $d$-dimen\-sio\-nal matrix of order $n$ in which entries $m_{\textbf{x}} = \frac{1}{(d-1)!}$  if $\textbf{x} = (x_1, \ldots, x_d)$ is a hyperedge in $\mathcal{G}$, and all other entries  $m_{\textbf{x}}$ are equal to $0$.  

Let us introduce several operations for multidimensional matrices. Given a $d$-dimensional matrix $A$ of order $n$  and $d-1$ vectors (functions) $f_1, \ldots, f_{d-1}$ of length $n$,  define the \textit{combined product} $A \cdot (f_1, \ldots, f_{d-1})$ to be the vector $f$ of length $n$ with entries
$$ f(x) = \sum\limits_{x_1, \ldots, x_{d-1} = 1}^{n}  a_{x, x_1, \ldots, x_{d-1}}   f_1(x_1) \cdots f_{d-1} (x_{d-1}). $$
Note that if for all $2 \leq i,j \leq d$ entries of the matrix $A$ remain unchanged under transpositions $(x_i, x_j)$ in indices, then the combined product  $A \cdot (f_1, \ldots, f_{d-1})$ does not depend on the order of the vectors $f_1, \ldots, f_{d-1}$.   In particular, this holds for the adjacency matrices of hypergraphs. 

 Let $A$ be a $d$-dimensional matrix of order $n$  and let $B$ be a $t$-dimensional matrix of the same order.   Define the \textit{product} $A \circ B$ to be the $((d -1) (t -1) +1)$-dimensional  matrix $C$ of order $n$ with entries 
 $$c_{x, \textbf{y}^2, \ldots, \textbf{y}^{d}} = \sum\limits_{x_2= 1}^n \cdots \sum\limits_{x_d= 1}^n  a_{x, x_2, \ldots, x_{d}} \cdot b_{x_2, \textbf{y}^2}\cdots b_{x_{d}, \textbf{y}^{d}},$$
 where indices $\textbf{y}^{2}, \ldots, \textbf{y}^{d} \in X^{t-1}$ and $x \in X$.  In particular, if $A$ is a $d$-dimensional matrix of order $n$ and $f$ is a vector of length $n$, then the combined product $A \cdot (f, \ldots, f)$ of $d-1$ copies of $f$ coincides with the product $A \circ f$.
 
We say that $\theta \in \mathbb{C}$ is an \textit{eigenvalue} of a $d$-dimensional matrix of order $n$  if there exists a \textit{eigenvector} (or \textit{eigenfunction}) $f \neq (0, \ldots, 0)$ such that $A \circ f = \theta (\mathbb{I} \circ f) $. Here $\mathbb{I}$ is the \textit{$d$-dimensional identity matrix}, whose entries satisfy $i_{\textbf{x}} = 1$ if $x_1 = \cdots = x_d$ and $i_{\textbf{x}} = 0$ otherwise.   In particular,  $\mathbb{I} \circ f$ is a function $f^{d-1}$. 

If $H$  is a $d$-uniform hypergraph with $d$-dimensional adjacency matrix $M$, then the eigenvalues of $M$ are called the \textit{eigenvalues} of  $H$, and the eigenvectors of $M$ are called the \textit{eigenvectors} of $H$.
 
Following~\cite{my.hyperperfcolor}, we provide several results on perfect colorings of hypergraphs. 

Let $\mathcal{G} (X,E)$ be a hypergraph.  A surjective function $f: X \rightarrow  C$, where $|C| = k$,  is called a \textit{coloring}  of $\mathcal{G}$ into $k$ colors (or a \textit{$k$-coloring}). Unless otherwise specified, the set of colors $C$ is $\{ 1, \ldots, k\}$.   A coloring $f$ determines a partition $P = (P_1, \ldots, P_k)$ of the vertex set $X$, where $P_i = \{ x \in X : f(x) = i\}$. 

 Given a coloring $f$ and a hyperedge $e$, the \textit{color range} $f (e)$ is the multiset of colors of vertices incident to $e$: $f (e) = \{f (x)|x \in e\}$.  A coloring $f$ of a hypergraph $\mathcal{G}$ is said to be \textit{perfect} if for every pair of vertices $x$ and $y$ of the same color, the multisets $\{ f(e) | x \in e  \}$ and   $\{ f(e) | y \in e  \}$ of color ranges of incident hyperedges coincide.  The partition  $(P_1, \ldots, P_k)$ corresponding to a perfect coloring $f$ is called an equitable partition.   Equivalently, a coloring is perfect if and only if, for every color range $\textbf{i} = \{ i_1, i_2, \ldots, i_d \}$, the number of hyperedges of color range $\mathbf{i}$ incident to a vertex $x$ of color $i_1$ does not depend on the choice of $x$. We denote these numbers by $v_{i_1,\textbf{i}}$. Figure 1 illustrates two perfect 2-colorings of the Fano plane, viewed as a hypergraph whose hyperedges are its lines.
 \begin{center}
 \includegraphics[width=0.3\linewidth]{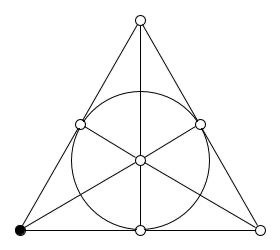}  ~~~
   \includegraphics[width=0.3\linewidth]{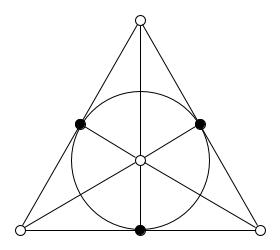} 
   
Figure 1. Perfect $2$-colorings of the Fano plane hypergraph. 
 \end{center}

For every $k$-coloring $f$ of a hypergraph $\mathcal{G} (X,E)$ on $n$ vertices, we define the \textit{color matrix} $F$ of size $n \times k$ such that the $(x,i)$-entry of $F$ equals  $1$  if $f(x) = i$ and  equals  $0$ otherwise.  In particular, the $i$-th column $F_i$ of  $F$ is the indicator function of the color $i$, $F = (F_1, \ldots, F_{k})$. We often identify  a coloring $f$ with its color matrix $F$.  

 The  following matrix characterization of perfect colorings of hypergraphs and the definition of the parameter matrix of a perfect coloring were given in~\cite{my.hyperperfcolor}.   

\begin{theorem} \cite[Theorem 3.4]{my.hyperperfcolor}  \label{adjparam}
Let $\mathcal{G}$ be a $d$-uniform hypergraph with $d$-dimensional adjacency matrix $M$. Then a coloring  $F$ of $\mathcal{G}$  into $k$ colors is perfect  if and only if there exists the $d$-dimensional parameter matrix $S$ of order $k$ such that $M \circ F = F \circ S $. Moreover,  the entries $s_{\textbf{i}}$, $\textbf{i} = (i_1, i_2, \ldots, i_{d} )$, of $S$ are given by
$$s_{\textbf{i}}  = v_{i_1,\textbf{i}} \cdot {d-1 \choose  d_1,  \ldots, d_k}^{-1},$$
where every color $j$ appears in the multiset $ \{i_2, \ldots, i_d \}$ exactly $d_j$ times, and  ${d-1 \choose  d_1,  \ldots, d_k} = \frac{(d-1)!}{d_1! \cdots d_k!}$ is the multinomial coefficient.
\end{theorem}

This theorem implies the following necessary condition on the existence of perfect colorings in hypergraphs: if there exists a perfect coloring of a hypergraph with the parameter matrix $S$,  then each  eigenvalue of $S$ is an eigenvalue of the adjacency matrix of the hypergraph.

\begin{utv} \cite[Theorem 3.11]{my.hyperperfcolor} \label{kvaluedeigen}
Let  $\mathcal{G}$ be a $d$-uniform hypergraph with adjacency matrix $M$.
If   $F$ is a perfect coloring of $\mathcal{G}$ with the parameter matrix $S$, then for every eigenvalue $\theta$ and corresponding eigenvector $g$ for the matrix $S$,   there exists the same eigenvalue $\theta$ and eigenvector $Fg$  of  the matrix $M$. 
\end{utv}

Finally, we introduce the convolution matrix and recall some properties of the Fourier transform.

Let $G$ be an abelian group of order $n$.  For a subset $V \subseteq G$, let $\chi_V$ denote the \textit{indicator function} of $V$:   $\chi_V (x) = 1$ if $x \in V$ and  $\chi_V (x) = 0$ otherwise. If $V = \{ x\}$, then we write $\chi_{x}$ instead of $\chi_{\{ x\}}$. Let  $\textbf{1}$ be the all-one function on $G$: $\textbf{1} (x) = 1$ for all $x \in G$. 

Given functions $f, g: G \rightarrow \mathbb{C}$,  their \textit{convolution} is defined by
$$ (f * g) (x) = \sum\limits_{y \in G} f(y) g(x-y).$$

Define the  \textit{convolution matrix $A_G$} of the group $G$  to be the $3$-dimensional   $(0,1)$-matrix   of order $n$ with entries $a_{x,y,z} = 1$  if $x=y+z$ and $a_{x,y,z} = 0$ otherwise. Similarly, we define  the  \textit{$t$-times convolution matrix $A^{(t)}_G$} of the group $G$ to be the $(t+2)$-dimensional   $(0,1)$-matrix of order $n$ such that $a_{\textbf{x}} = 1$ if and only if $x_1=x_2+ \cdots +x_{t+2}$.  In particular, the convolution  matrix is the   $1$-times  convolution matrix, that is $A_G^{(1)} = A_G$.  In Section~\ref{relationsec}, we show that the convolution matrix $A_G$  is in fact a coordinate representation of the convolution operator.  We also note that the set of nonzero entries of the matrix $A^{(t)}_{G}$ forms a linear $n$-ary MDS code of length $t+2$ and distance $2$.  If $G$ is the additive group of the space $\mathbb{F}_2^n$, we denote the convolution matrices as $A$ and $A^{(t)}$ instead of $A_{\mathbb{F}_2^n}$ and $A^{(t)}_{\mathbb{F}_2^n}$.  For example, the convolution matrix of the additive group of the space $\mathbb{F}_2^2$ is the following $3$-dimensional matrix of order $4$:
 $$
 \left( 
\begin{array}{cccc|cccc|cccc|cccc}
1 & 0 & 0 & 0 & 0 & 1 & 0 & 0 & 0 & 0 & 1 & 0 & 0 & 0 & 0 & 1 \\ 
0 & 1 & 0 & 0 & 1 & 0 & 0 & 0 & 0 & 0 & 0 & 1 & 0 & 0 & 1 & 0 \\ 
0 & 0 & 1 & 0 & 0 & 0 & 0 & 1 & 1 & 0 & 0 & 0 & 0 & 1 & 0 & 0 \\ 
0 & 0 & 0 & 1 & 0 & 0 & 1 & 0 & 0 & 1 & 0 & 0 & 1 & 0 & 0 & 0 \\ 
\end{array}  
 \right).
 $$

A function $f: \mathbb{F}_2^n \rightarrow \mathbb{F}_2$ is called a Boolean function.  Define the \textit{Fourier transform} of a Boolean function $f$ by
$$\widehat{f} (x) = \frac{1}{2^{n/2}} \sum\limits_{z \in \mathbb{F}_2^n} f(z) (-1)^{(x,z)}, $$
where $(x,z) = x_1 z_1 + \cdots + x_n z_n$.  The normalization for the Fourier transform is chosen so that it is an isometry.

We will also need the following well-known properties of the Fourier transform that can be found, for example, in~\cite[Section 2.3]{carle.boolforcrypt} or~\cite[Section 4]{TaoVu.addcomb}.

\begin{utv}   \label{fourierutv} \item
\begin{enumerate}
\item $\widehat{\widehat{f}} = f$ for every Boolean function $f: \mathbb{F}_2^n \rightarrow \mathbb{F}_2$. 
\item  $\widehat{f + g } = \widehat{f } + \widehat{g}  $ for all Boolean functions  $f, g: \mathbb{F}_2^n \rightarrow \mathbb{F}_2$.
\item  $\widehat{f * g } = 2^{n/2} (\widehat{f} \cdot \widehat{g})$ for all Boolean functions  $f, g: \mathbb{F}_2^n \rightarrow \mathbb{F}_2$.
\item $\widehat{\textbf{1}} = 2^{n/2} \chi_{\overline{0}}$ and $\widehat{\chi_{\overline{0}}} = \frac{1}{2^{n/2}} \textbf{1}$ for functions $\textbf{1}$ and $\chi_{\overline{0}}$ over $\mathbb{F}_2^n$.
\item If $U \subseteq \mathbb{F}_2^n$ is a linear subspace of dimension $k$ in $\mathbb{F}_2^n$, then $\widehat{\chi_U} = 2^{k - n/2} \chi_{U^\perp}$, where $U^\perp = \{ x \in \mathbb{F}_2^n : (x,y) = 0 \mbox { for all } y \in U\}$. 
\end{enumerate}
\end{utv}

\section{Eigenvalues of convolution matrices of $\mathbb{F}_2^n$ and $\mathbb{F}_3^n$} \label{eigensec} 

In this section, we find some eigenvalues of  the convolution matrices of the groups $\mathbb{F}_2^n$ and $\mathbb{F}_3^n$ with the help of their characteristic polynomials and the Kronecker product. 

Qi~\cite{qi.eigentensor}  proved that for every $d$-dimensional matrix $A$ of order $n$, there exists a \textit{characteristic polynomial} $\varphi_{A}$ of degree $n(d-1)^{n-1}$ such that a number $\theta \in \mathbb{C}$ is an eigenvalue of $A$ if and only if $\theta$ is a root of $\varphi_{A}$. The existence of such a polynomial is based on  the theory of determinants of multidimensional matrices and resultants of multilinear systems developed by Gelfand, Kapranov, and Zelevinsky  in~\cite{GeKaZe.multdet}. 

If $A$ is a $d$-dimensional matrix of order $n_1$ and $B$ is a $d$-dimensional  matrix of order $n_2$, then the \textit{Kronecker product} $A\otimes B$ of matrices $A$ and $B$ is the $d$-dimensional matrix $C$ of order $n_1 n_2$ with entries $c_{\textbf{z}} = a_{\textbf{x}} b_{\textbf{y}},$ where  $z_i = (x_i - 1) n_2 + y_i$ for each $i = 1, \ldots, d$. In particular, if $A$ and $B$ are matrices whose entries are $0$ and $1$, then their Kronecker product is also a $(0,1)$-matrix $C$ whose $1$-entries have indices $( (x_1 - 1) n_2 + y_1, \ldots, (x_d - 1) n_2 + y_d )$ for which $a_{x_1, \ldots, x_d} = 1$ and  $b_{y_1, \ldots, y_d} = 1$. It can be verified by induction that  the convolution matrix $A_{\mathbb{F}_p^n}$ of   $\mathbb{F}_p^n$ is  the Kronecker product of $n$ copies of the convolution matrix $A_{\mathbb{F}_p}$.

The following result on eigenvalues of the Kronecker product of matrices was proved in~\cite{shao.tensprod}.

\begin{utv} \cite[Theorem 3.4]{shao.tensprod} \label{eigenKronprod}
Let $A$ be a $d$-dimensional matrix of order $n_1$ and let $B$ be a $d$-dimensional  matrix of order $n_2$. If $\theta$ is an eigenvalue of $A$ and $\theta'$ is an eigenvalue of $B$, then $\theta \theta'$ is an eigenvalue of the Kronecker product $A \otimes B$.
\end{utv}

Using this property, we find some eigenvalues of the convolution matrix $A_{\mathbb{F}_2^n}$.

 \begin{theorem} \label{F2nvalues}
Let $A_{\mathbb{F}_2^n}$  be   the  $3$-dimensional convolution matrix of   $\mathbb{F}_2^n$.  Then for all nonnegative integers $k_0, k_1, k_2$ with $k_0 + k_1 + k_2 \leq n $, the value $2^{k_0 - k_1 - k_2} (-1 + i\sqrt{7})^{k_1} (-1 - i\sqrt{7})^{k_2} $ is an eigenvalue of  $A_{\mathbb{F}_2^n}$.
\end{theorem}

\begin{proof}
Recall that the convolution matrix $A_{\mathbb{F}_2^n}$ is the $n$-times Kronecker product of the matrix $A_{\mathbb{F}_2}$.  The  convolution matrix $A_{\mathbb{F}_2}$ of $\mathbb{F}_2$  is a $3$-dimensional matrix of order $2$ of the form
$$A_{\mathbb{F}_2} = \left( \begin{array}{cc|cc} 
1 & 0 & 0 & 1 \\ 
0 & 1 & 1 & 0
\end{array} \right).$$
Using the calculation of Morozov and Shakirov~\cite{MorShak.resultform} for the resultant  of a system of equations  given by an arbitrary  $3$-dimensional matrix of oder $2$, we find  the characteristic polynomial $\varphi (\theta) = Res (A - \theta \mathbb{I})$ of  $A_{\mathbb{F}_2}$:
$$\varphi (\theta) =  \theta^4 - 2 \theta^3 + \theta^2 - 4 \theta + 4 = (1 - \theta) (2 - \theta) (\theta^2 + \theta + 2).$$

Consequently, the   matrix $A_{\mathbb{F}_2}$ has eigenvalues $1$, $2$, and $\frac{1}{2} (-1 \pm i \sqrt{7})$. Thus, by Proposition~\ref{eigenKronprod}, the convolution matrix $A_{\mathbb{F}_2^n}$ has eigenvalues  $2^{k_0 - k_1 - k_2} (-1 + i\sqrt{7})^{k_1} (-1 - i\sqrt{7})^{k_2} $  for all nonnegative integers $k_0, k_1, k_2$ such that  $k_0 + k_1 + k_2 \leq n $.
\end{proof}

Acting similarly, we find some eigenvalues of the convolution matrix $A_{\mathbb{F}_3^n}$.

\begin{theorem}
Let $A_{\mathbb{F}_3^n}$  be the $3$-dimensional convolution matrix of $\mathbb{F}_3^n$. Then for all nonnegative integers $k_0, k_1, k_2$ with  $k_0 + k_1 + k_2 \leq n $, the value $3^{k_0}  ( i\sqrt{3})^{k_1} (- i\sqrt{3})^{k_2} $ is an eigenvalue of $A_{\mathbb{F}_3^n}$.
\end{theorem}

\begin{proof}
As in the previous  proof,  $A_{\mathbb{F}_3^n}$ is the $n$-times Kronecker product of the matrix $A_{\mathbb{F}_3}$.  The  convolution matrix $A_{\mathbb{F}_3}$ of  $\mathbb{F}_3$  has the form
$$ A_{\mathbb{F}_3}= \left( \begin{array}{ccc|ccc|ccc} 
1 & 0 & 0 & 0 & 0 & 1 & 0 & 1 & 0  \\ 
0 & 0 & 1 & 0 & 1 & 0 & 1 & 0 & 0  \\ 
0 & 1 & 0 & 1 & 0 & 0 & 0 & 0 & 1  \\ 
\end{array} \right).$$

Let $\mathcal{E}$ denote the adjacency matrix of the $3$-uniform hypergraph $H$ on $3$ vertices with a single hyperedge:
$$ \mathcal{E} = \left( \begin{array}{ccc|ccc|ccc} 
0 & 0 & 0 & 0 & 0 & \nicefrac{1}{2} & 0 & \nicefrac{1}{2} & 0  \\ 
0 & 0 & \nicefrac{1}{2} & 0 & 0 & 0 & \nicefrac{1}{2} & 0 & 0  \\ 
0 & \nicefrac{1}{2} & 0 & \nicefrac{1}{2} & 0 & 0 & 0 & 0 & 0  \\ 
\end{array} \right).$$
Note that  $A_{\mathbb{F}_3}= 2 \mathcal{E} + \mathbb{I}$. Since  the characteristic polynomial $\varphi_{A_{\mathbb{F}_3}}  (\theta)$ of  $A_{\mathbb{F}_3}$  is the resultant  of the system $ (A_{\mathbb{F}_3} - \theta \mathbb{I}) f = 0 $ (see~\cite{qi.eigentensor}), $\varphi_{A_{\mathbb{F}_3}}$  can be obtained from the characteristic polynomial $\varphi_{2\mathcal{E}} (\theta)$ of $2 \mathcal{E}$ by replacing $\theta$ with $\theta - 1$.

The characteristic polynomial $\varphi_{\mathcal{E}} (\theta)$ for the adjacency matrix $\mathcal{E}$ was found in~\cite{CoopDut.hyperspec} by Cooper and Dutle:
$$\varphi_{\mathcal{E}} (\theta) = \theta^3 (\theta^3 - 1)^3 =  \theta^3 (\theta - 1)^3 (\theta^2 + \theta + 1)^3.  $$
The eigenvalues of  the  matrix $2 \mathcal{E}$ are those of $\mathcal{E}$ multiplied by $2$. Consequently, its characteristic polynomial is
$$\varphi_{2\mathcal{E}} (\theta)  =  \theta^3 (\theta - 2)^3 (\theta^2 + 2 \theta + 4)^3.  $$
 Replacing  $\theta$ with $\theta - 1$, we obtain
 $$\varphi_{A_{\mathbb{F}_3}} (\theta)  =  (\theta - 1)^3 (\theta - 3)^3 (\theta^2 + 3)^3 . $$
 Thus $1$, $3$, and  $\pm i \sqrt{3}$ are  the eigenvalues of $A_{\mathbb{F}_3}$. By Proposition~\ref{eigenKronprod}, the convolution matrix $A_{\mathbb{F}_3^n}$ has eigenvalues $3^{k_0}  ( i\sqrt{3})^{k_1} (- i\sqrt{3})^{k_2} $   for all nonnegative integers $k_0, k_1, k_2$ such that  $k_0 + k_1 + k_2 \leq n $.
 
\end{proof}

\section{Relations between convolution matrices, subspace hypergraphs, and perfect colorings} \label{relationsec}

We begin this section by showing  that the convolution matrix does define the convolution: the convolution of a collection of functions coincides with their combined product with the convolution matrix.

\begin{lemma} \label{convmatr}
Let $G$ be a finite abelian group, let $A^{(t)}_G$ be the $t$-times convolution matrix of the group $G$, and let $f_1, \ldots, f_{t+1} : G \rightarrow \mathbb{C}$. Then 
$$A^{(t)}_G  \cdot (f_1, \ldots, f_{t+1}) = f_1 *   \cdots   * f_{t+1}.$$
\end{lemma}

\begin{proof}
From the definition of convolution, for each $x \in G$ we have
$$( f_1 *   \cdots   * f_{t+1}) (x) = \sum\limits_{x_1, \ldots, x_{t} \in G} f_1(x_1)  \cdots f_{t}(x_{t}) \cdot f_{t+1} (x - x_1 - \ldots - x_{t}).$$

Using the definition of the combinded product and  the fact that the entry $a_{x_0, x_1, \ldots x_{t+1}} $ of  the  convolution matrix $A^{(t)}_G$  is equal to $1$ if and only if $x_0 = x_1+ \cdots +x_{t+1}$ in $G$, we obtain
\begin{gather*}
A^{(t)}_G \cdot (f_1,   \ldots, f_{t+1})  (x)=\sum\limits_{x_1, \ldots, x_{t+1} \in G} a_{x,x_1, \ldots, x_{t+1}} f_1(x_1) \cdots  f_{t+1}( x_{t+1})= \\
\sum\limits_{x_1, \ldots, x_{t} \in G} f_1(x_1) \cdots f_{t} (x_t)  \cdot f_{t+1}( x - x_1 - \ldots - x_{t} )  =( f_1 *   \ldots   * f_{t+1}) (x) .
\end{gather*}
\end{proof}

Since the combined product $A \cdot (f, \ldots, f)$ of a multidimensional matrix $A$ and a functions $f$ coincides with the product $A \circ f$, we obtain the following property.

\begin{sled} \label{convaseigen}
Let $G$ be a finite abelian group, let $f: G \rightarrow \mathbb{C}$, and let $A^{(t)}_G$ be the $t$-times convolution matrix of the group $G$. Then $A^{(t)}_G \circ f = f *   \cdots   * f$,  that is, the convolution of $t+1$ copies of the function $f$.
\end{sled}

We also show that  the  product of convolution matrices is a convolution matrix of higher dimension.

\begin{utv}
If $A^{(t)}_G$  and $A^{(s)}_G$  are  $t$-times and $s$-times convolution matrices of a finite abelian group $G$, respectively, then $ A^{(t)}_G  \circ A^{(s)}_G = A^{(ts + t + s)}_G$.  In particular, $A^{(1)}_G \circ A^{(1)}_G = A^{(3)}_G$. 
\end{utv}

\begin{proof}
By the definition of the product of multidimensional matrices, $ A^{(t)}_G  \circ A^{(s)}_G$ is a $(ts + t + s +2)$-dimensional matrix $C$ with entries 
$$ c_{x, \textbf{y}^2, \ldots, \textbf{y}^{t+2}} = \sum\limits_{x_2= 1}^n \cdots \sum\limits_{x_{t+2}= 1}^n  a^{(t)}_{x, x_2, \ldots, x_{t+2}} \cdot a^{(s)}_{x_2, \textbf{y}^2}\cdots a^{(s)}_{x_{t+2}, \textbf{y}^{t+2}},$$
 where indices $\textbf{y}^{2}, \ldots, \textbf{y}^{t+2} \in G^{s+1}$ have length $s+1$ and $x \in G$.  By the definition of the convolution matrices, $a^{(t)}_{x, x_2, \ldots, x_{t+2}} = 1$ if $x = x_2 + \cdots + x_{t+2}$ and $a^{(t)}_{x, x_2, \ldots, x_{t+2}} = 0$ otherwise, and  for all $i \in \{ 2, \ldots, t+2 \}$ we have that  $a^{(s)}_{x_i, \textbf{y}^i} = 1$ if $x_i = y^i_1 + \cdots + y^i_{s+1}$ and   $a^{(s)}_{x_i, \textbf{y}^i} = 0$  otherwise. Consequently, $c_{x, \textbf{y}^2, \ldots, \textbf{y}^{t+2}} = 1$ if $x=y^2_1 + \cdots + y^2_{s+1}  + \cdots +  y^{t+2}_1 + \cdots + y^{t+2}_{s+1} $ and $c_{x, \textbf{y}^2, \ldots, \textbf{y}^{t+2}} = 0$ otherwise, that is equivalent to that $ A^{(t)}_G  \circ A^{(s)}_G = A^{(ts + t + s)}_G$.
\end{proof}

Next, we express combined product involving convolution matrices of $\mathbb{F}_2^n$   in terms of combined product involving the adjacency matrices of the  subspace hypergraphs $\mathcal{H}_n$ and $\mathcal{D}_n$.

\begin{utv} \label{convhyper3comb}
Let $A$ be the convolution matrix of $\mathbb{F}_2^n$, let $M_{H}$ be the adjacency matrix of the subspace hypergraph $\mathcal{H}_n$, $f , g : \mathbb{F}_2^n \rightarrow \mathbb{C}$, and let $f'$ and $g'$ be the restrictions of $f$ and $g$ to $\mathbb{F}_2^n \setminus \{ \overline{0}\}$.  Then for every $x \neq \overline{0}$ we have
$$(A \cdot(f,g))  (x) = 2 (M_{H} \cdot  (f' ,g')) (x) +  f(\overline{0}) g(x)  +f(x)   g(\overline{0}) $$
and
$$(A \cdot(f,g))   (\overline{0}) = \sum\limits_{x \in \mathbb{F}_2^n} f(x) g(x).$$
\end{utv}

\begin{proof}
By the definition of the  convolution matrix,  $a_{x, y, z} = 1$ if   $x = y +z$, and   $a_{x, y, z} = 0$ otherwise.  Note that in the group $\mathbb{F}_2^n$ it holds  $x = - x$ for all $x$, and  the equality $x + y = \overline{0}$ is equivalent to $x = y$.  Then for all $x \neq \overline{0}$ we have

\begin{gather*} 
(A \cdot(f,g))  (x)  = \sum\limits_{y, z \in \mathbb{F}_2^n} a_{x, y, z} f(y) g(z)  = \sum\limits_{y \in \mathbb{F}_2^n} f(y) g(x + y) =  \\
 \sum\limits_{y \neq \overline{0},  x  }  f(y) g(x + y) +  f(\overline{0}) g(x)  + f(x)  g(\overline{0}) =  
 2 (M_{H} \cdot  (f' ,g')) (x) +  f(\overline{0}) g(x)  +f(x)   g(\overline{0}). 
\end{gather*}
If $x = \overline{0}$, then
$$
(A \cdot(f,g))   (\overline{0}) =  \sum\limits_{x, y \in \mathbb{F}_2^n} a_{\overline{0}, x, y} f(x) g(y)  = \sum\limits_{x \in \mathbb{F}_2^n} f(x) g(x).
$$
\end{proof}

\begin{sled} \label{3conhypeigen}
Let $A$ be the convolution matrix of $\mathbb{F}_2^n$, $M_{H}$ be the adjacency matrix of the subspace hypergraph $\mathcal{H}_n$,  $f: \mathbb{F}_2^n \rightarrow \mathbb{C}$, and let $f'$ be the restriction of $f$ to $\mathbb{F}_2^n \setminus \{ \overline{0}\}$.
\begin{itemize}
\item If $f$ is an eigenvector of $A$ corresponding to an eigenvalue $\theta$ and $f(\overline{0}) = 0$, then $f'$ is an eigenvector for $M_{H}$ with an eigenvalue  $ \theta /2 $.
\item  If $f$ is an eigenvector of $A$ corresponding to an eigenvalue $\theta$ and  $f : \mathbb{F}_2^n \rightarrow \{0, 1 \}$, then $f'$ is  an eigenvector for $M_{H}$ with an eigenvalue $\theta /2  -  f(\overline{0})$.
\end{itemize} 
\end{sled}

\begin{proof}
If $f$ is an eigenfunction of $A$ corresponding to an eigenvalue $\theta$, then $A \circ f = \theta (\mathbb{I} \circ f)$, where $\mathbb{I}$ is the $3$-dimensional identity matrix. 
From Proposition~\ref{convhyper3comb}, for every $x \neq \overline{0}$  we have 
$$A \circ f  (x) = 2 (M_{H}  \circ f') (x) + 2  f(\overline{0}) f(x).$$
Hence, 
$$ M_H \circ f' =   \theta/2 ( \mathbb{I} \circ  f') - f(\overline{0}) f' .$$
If $f(\overline{0}) = 0$, then $\theta / 2$ is the eigenvalue of $M_H$ for the eigenfunction function $f'$.

Since for each  $x \neq \overline{0}$  we have $(\mathbb{I} \circ f') (x) =(f'(x))^2$, this value coincides with  $ f'(x)$ if and only if $f'(x) $ is equal to $0$ or $1$. Consequently, if $f$ takes only values $0$ and $1$, then $f'$ is  the eigenfunction of $M_{H}$ corresponding to the eigenvalue $\theta/2  -  f(\overline{0})$.
\end{proof}

We obtain a similar result for the $2$-times convolution matrix and the adjacency matrix of the subspace hypergraph $\mathcal{D}_n$.

\begin{utv} \label{2convhyper4}
Let $A^{(2)}$ be the  $2$-times convolution matrix of $\mathbb{F}_2^n$,  $M_{D}$  be the adjacency matrix of the subspace hypergraph $\mathcal{D}_n$, and let $f: \mathbb{F}_2^n \rightarrow \mathbb{C}$.  Then for all $x \in \mathbb{F}_2^n$ we have
$$(A^{(2)} \circ f) (x) = 6 (M_{D} \circ f) (x)  + 3 (\sum\limits_{y  \neq x} f^2(y)) f(x) +  (\mathbb{I} \circ f) (x). $$
\end{utv}

\begin{proof}
By the definition of the $2$-times convolution matrix $A^{(2)}$,  its entry $a_{x, y, z, v} = 1$ if $x = y +z + v$ and  $a_{x, y, z, v} = 0$ otherwise.  Note that  if   $x,y,z,v   \in \mathbb{F}_2^n$   satisfy  $x = y +z + v$ and two of them  coincide (for example, $x = y$), then the remaining two elements  also  coincide ($z = v$).  

Given $x \in \mathbb{F}_2^n$, let $R_x$  denote the set of all ordered pairs $(y,z)$ such that $x, y, z$, and $x + y + z$ are pairwise distinct.   Then for every  $x \in \mathbb{F}_2^n$ we have

 \begin{gather*}
 (A^{(2)} \circ f) (x)  = \sum\limits_{y,z,v \in \mathbb{F}_2^n} a_{x,y,z,v} f(y) f(z) f(v) =  \\
  \sum\limits_{(y,z) \in R_x} f(y) f(z) f(x + y +z)  +  3 \sum\limits_{y \neq x } f(y)^2  f(x)  + f^3 (x) =\\
    6 (M_{D} \circ f) (x)  + 3 (\sum\limits_{y  \neq x} f^2(y)) f(x) +  (\mathbb{I} \circ f) (x),
\end{gather*}
where $\mathbb{I}$ is the $4$-dimensional identity matrix.
\end{proof}

Similarly to Corollary~\ref{3conhypeigen}, the above statement can also be used to show that certain eigenfunctions of the $2$-times convolution matrix give rise to eigenfunctions of the subspace hypergraph and vice versa.

Next, we use the obtained results to show that  perfect colorings  of the subspace hypergraph $\mathcal{H}_n$ can be described  in terms of the convolution of the indicator functions of colors. In particular, we show that  the  column space of the  color matrix is an algebra with convolution as the multiplication.

\begin{lemma} \label{colorasconv}
Let  $M_{H}$ be the adjacency matrix of the subspace hypergraph $\mathcal{H}_n$. A $k$-coloring   with  color matrix $F = (F_1, \ldots, F_k)$  is a  perfect coloring of $\mathcal{H}_n$ with the parameter matrix $S$ if and only if for all $j,\ell \in \{ 1, \ldots, k\}$ it holds $  F_j * F_\ell  =  \sum\limits_{i = 1}^k  2  s_{i,j, \ell}  F_{i}   $. 
\end{lemma}

\begin{proof}
Recall that the vertex set of the  subspace hypergraph $\mathcal{H}_n$ is $\mathbb{F}_2^n \setminus \{ \overline{0}\}$.  Without loss of generality, we extend the indicator functions of the colors $F_i$ from $\mathbb{F}_2^n \setminus \{ \overline{0}\}$ to $\mathbb{F}_2^n$ by setting $F_i (\overline{0}) = 0$ for all $i \in \{1, \ldots, k \}$. This modification does not change the value $(F_j * F_\ell)  (x)$ for all $x \neq \overline{0}$. 

Consider a vertex $x \neq \overline{0}$ of color $i$ in the coloring $F$.   From Lemma~\ref{convmatr} and Proposition~\ref{convhyper3comb},  for all $j, \ell \in \{ 1, \ldots, k\}$ (either or both of which may be equal to $i$) we have 
$$ ( F_j * F_\ell)  (x) = A \cdot (F_j, F_\ell)  (x)= 2 (M_{H} \cdot  (F_j ,F_\ell)) (x). $$

If  $j \neq\ell$, then by the definition of the adjacency matrix of $\mathcal{H}_n$ and $F_i(x) = 1$, we obtain
$$ (M_{H} \cdot  (F_j ,F_\ell)) (x)  = \sum\limits_{(y,  z):  (x,y,z)  \in E(\mathcal{H}_n)}  1/2 \cdot  F_{j} (y) F_{\ell} (z) =   v_{i, j, \ell} (x) / 2 =  v_{i, j, \ell} (x) / 2 \cdot F_i(x),$$
where $v_{i,j,\ell} (x)$ denotes the number of hyperedges of color range $\{ i,j,\ell \}$  incident to the vertex $x$, and $E(\mathcal{H}_n)$ is the hyperedge set of $\mathcal{H}_n$. By  Theorem~\ref{adjparam}, the coloring $F$ is perfect if and only if the number  $v_{i, j, \ell} (x)$  does not depend on the choice of $x$ and equals  $2 s_{i, j, \ell}$.

Similarly, if $j = \ell$, then
$$ (M_{H} \cdot  (F_j ,F_j)) (x)  = \sum\limits_{(y,  z):  (x,y,z)  \in E(\mathcal{H}_n)}  1/2 \cdot  F_{j} (y) F_{j} (z) =   v_{i, j, j}  (x) =  v_{i, j, j}  (x) F_i(x) ,$$
where $v_{i,j,j} (x)$ denotes the number of hyperedges of color range $\{ i,j,j \}$  incident to  $x$. Again, by  Theorem~\ref{adjparam}, the coloring $F$ is perfect if and only if the number  $v_{i, j, j} (x)$  does not depend on  $x$ and equals   $s_{i, j, j}$. 

Therefore, $F$ is a perfect coloring with the parameter matrix $S$ if and only if for all $j,\ell \in \{ 1, \ldots, k\}$ it holds $  F_j * F_\ell  =   \sum\limits_{i = 1}^k   2 s_{i,j, \ell}  F_{i}   $. 
\end{proof}

Note that for the $4$-uniform subspace hypergraph $\mathcal{D}_n$ one can also state a similar condition for a coloring $F$ to be perfect in terms of the double convolution of $F_j$, $F_\ell$, and $F_m$. However, in that case the coefficients in the linear combinations of $F_i$ do not admit such a simple expression.

It is well known~\cite[Corollary 1.9]{ShiSole.CRCinDRG} that every $2$-valued eigenfunction of a regular graph corresponds to  a perfect $2$-coloring.  By Theorem~\ref{adjparam}, to every perfect $2$-coloring of a uniform hypergraph $\mathcal{G}$ we can associate a $2$-valued eigenfunction of its adjacency matrix $M$. Indeed, $F$ is the color matrix of a perfect $2$-coloring with the parameter matrix $S$ if and only if  $M \circ F = F \circ S$. Then for every eigenfunction $g = (\alpha , \beta)$, with $\alpha \neq \beta$, of the  matrix $S$, the function $Fg$ is a $2$-valued eigenfunction of the matrix $M$, where the distinct values  of $Fg$ correspond to the two colors in the perfect coloring (see also~\cite[Theorem 3.10]{my.hyperperfcolor}).

Unlike the graph case, for hypergraphs  this statement cannot be naturally reversed. We prove this for the  family of totally regular hypergraphs, which includes subspace hypergraphs $\mathcal{H}_n$ and $\mathcal{D}_n$; the general case remains open.

Let $\mathcal{G}(X,E)$ be a $d$-uniform hypergraph. Given a set $S \subseteq X$ of vertices of $\mathcal{G}$, the \textit{degree} $deg(S)$ of  $S$ is the number of hyperedges $e \in E$ such that  $S \subseteq E$.  We say that a $d$-uniform hypergraph $\mathcal{G}$ is \textit{totally regular}  (or \textit{totally $(r_1, \ldots, r_{d-1})$-regular}) 
 if  there exist numbers $r_i \in \mathbb{N}$, $i = 1, \ldots, d-1$,  such that  for every subset $S$ of vertices with $|S| = i$  we have $deg(S) = r_i$. Totally regular hypergraphs appear in literature under different names. For example in~\cite{niki.anmethhyper}, $d$-unform totally regular hypergraphs are called  $(d-1)$-set regular.

Totally regular hypergraphs are in one-to-one correspondence with block designs.  A  $t$-$(n,k,\lambda)$-design is a collection of blocks of size $k$ over a ground set $X$ of size $n$ such that every $t$-element subset of $X$ is contained in exactly $\lambda$ blocks; $t$-$(n,k,1)$-designs are called Steiner systems $S(t,k,n)$.  The hyperedges of  a $d$-uniform totally regular hypergraph on $n$ vertices, with degree of $(d-1)$-element subsets  equal to $r_{d-1}$,  form the blocks of a $(d-1)$-$(n,d,r_{d-1})$-design.  In particular, from the divisibility conditions for the parameters of  $(d-1)$-$(n,d,r_{d-1})$-designs, it follows that all parameters $r_i$ of totally $(r_1, \ldots, r_{d-1})$-regular hypergraphs are propotional to $r_{d-1}$.  The set of  hyperedges of the subspace hypergraph $\mathcal{H}_n$ forms a linear Steiner triple system $S(2,3,2^{n}-1)$, and the set of  hyperedges of $\mathcal{D}_n$ forms a linear Steiner quadruple system $S(3,4,2^{n})$.  Consequently, the subspace hypergraphs $\mathcal{H}_n$ and $\mathcal{D}_n$ are totally regular.

We are now ready to prove that $2$-valued eigenfunctions  of  totally regular hypergraphs are their perfect $2$-colorings.

\begin{theorem} \label{eigencolor}
Let $d \geq 3$, and let $\mathcal{G} (X,E)$  be a $d$-uniform totally $(r_1, \ldots, r_{d-1})$-regular  hypergraph  on $n$ vertices  with the adjacency matrix $M$.  Assume that  $f$ is  a $2$-valued eigenvector of $M$ corresponding to an eigenvalue $\theta$ such that $f(x) \in \{ \alpha, \beta \}$ for all $x \in X$, where $\alpha \neq \beta$. Then $f$ is a perfect  $2$-coloring  of $\mathcal{G}$.
\end{theorem}

\begin{proof}
Let $n_\alpha$ and $n_\beta$ be the total number of vertices $x$ of color $\alpha$ and $\beta$, respectively, $n_\alpha + n_\beta = n$.  Consider an arbitrary vertex $x$ of color $\alpha$. Then  the equation $(M \circ f) (x) = \theta (\mathbb{I} \circ f) (x)$ is equivalent to 
\begin{equation} \label{eigeneq}
\alpha^{d-1} t_0^x + \alpha^{d-2} \beta t_{1}^x + \cdots + \beta^{d-1} t_{d-1}^x  = \theta \alpha^{d-1},
\end{equation}
where $t_{j}^x$ denotes the number of hyperedges $e \in  E$ incident to  $x$ such that $d-j$ vertices of $e$ have color $\alpha$ and  $j$ vertices of $e$  have color $\beta$.  Note that if $j > n_\beta$ or $ d - j > n_\alpha$, then $t^x_j = 0$ by the definition.  

Denote $J  = \min \{ n_\alpha, d-1 \}$.   For every $j  \in \{ 0, \ldots,  J-1\}$, consider an arbitrary collection of  distinct vertices  $x^{1}, \ldots, x^{j}$,  all different from $x$ and all of color $\alpha$.  Let $t_k^{x,x^1, \ldots, x^j}$ denote  the number of hyperedges $e$ incident to the vertices $x, x^{1}, \ldots, x^{j}$ such that $d-k$ vertices from $e$ have color $\alpha$ and  $k$ vertices from $e$  have color $\beta$. In particular, if $d-k < j +1$ or $d - k > n_\alpha$  or  $k > n_\beta$, then $t_k^{x,x^1, \ldots, x^j} = 0$.   Since $\mathcal{G}$ is totally regular, double counting  the number of all hyperedges  containing the set $x, x^{1}, \ldots, x^{j}$  yields
$$t_0^{x,x^1, \ldots, x^j} + \cdots + t_{d-1 - j}^{x,x^1, \ldots, x^j} = r_{j +1}. $$

Summing this equality over  all subsets  of $j$ distinct vertices  $x^{1}, \ldots, x^{j}$ different from $x$ and  of color $\alpha$,   we get
\begin{equation} \label{systemeq}
{d-1 \choose j} t_0^{x} +{d-2 \choose j} t_1^{x} + \cdots + {j \choose j} t_{d-1 - j}^{x} =  {n_\alpha - 1 \choose j} r_{j+1}, ~~ j = 0 ,\ldots, J -1,
\end{equation}
where $n_{\alpha}$ denotes the number of vertices of color $\alpha$.

Equations~(\ref{systemeq}) together with equation~(\ref{eigeneq}), form a system of $J+1$ linear equations in the variables $t_0^{x}$, \ldots, $t_{d-1}^x$.  Note that if $J = n_{\alpha} < d-1$, then we already know that $t^x_j = 0$  for all  $j <  d - n_{\alpha} $. So the number of equations in ths system coincides with the number of variables.   We show that this system is nondegenerate for $J = d-1$, for $J < d-1$ the reasoning is similar. 

The equations in system~(\ref{systemeq}) are linearly independent because their coefficient matrix is upper triangular.  To eliminate the terms  $ \alpha^{d-2} \beta t_{1}^x, \ldots, \beta^{d-1} t_{d-1}^x $ from  equation~(\ref{eigeneq}),  we take a suitable linear combination of the equations in~(\ref{systemeq}), multiplying the $j$-th equation by $- \beta^{d-1 - j} (\alpha - \beta)^j$ and add them to~(\ref{eigeneq}). The resulting coefficient of $t_0^x$  equals  $(\alpha - \beta)^{d-1}$ which is nonzero since $\alpha \neq \beta$.  Therefore, the  combined system consisting of (\ref{eigeneq}) and (\ref{systemeq}) is nondegenerate and has a unique solution.

Observe that the coefficients of this system depend only on the color of $x$, not on the specific choice of the vertex.  Hence the values $t_0^{x}$, \ldots, $t_{d-1}^x$, i.e., the number of incident hyperedges of a given color range incident  to a vertex $x$ of color $\alpha$, are uniquely determined and independent of the particular vertex $x$ of that color. 

Repeating the same argument for vertices of color $\beta$, we conclude that the coloring $f$ is perfect. 

\end{proof}

\section{Designs are perfect colorings and eigenfunctions of subspace hypergraphs} \label{designsec}

In this section we, apply the results obtained above to show that partial difference sets, bent and plateaued functions,  spreads and bent partitions correspond to perfect colorings or eigenfunctions of the subspace hypergraphs $\mathcal{H}_n$ and $\mathcal{D}_n$  or to eigenfunctions of the convolution matrices.

\subsection{Partial difference sets}

Let $G$ be a finite abelian group.  A subset $D\subseteq G$  is called a \textit{difference set} with parameters $(v,k,\lambda)$ if   $|G|=v$, $|D|=k$, and for every nonzero element $ x \in G$  there exist exactly $\lambda$ pairs of elements $y_1, y_2\in D$ such that $y_1-y_2=x$.
A subset  $D\subset G $, with $0 \not\in D$,  is called a \textit{partial difference set} with parameters  $(v,k,\lambda,\mu)$ if  $|G|=v$, $|D|=k$, for every $ x\in D$ there are  exactly $\lambda$ pairs $y_1, y_2\in D$ such that    $y_1-y_2=x$, and for every nonzero $ x\in G\setminus D$  there are exactly  $\mu$ pairs $y_1, y_2\in D$ such that  $y_1-y_2=x$. Observe that a partial difference set with parameters $(v,k,\lambda,\lambda)$ is precisely  a difference set with parameters $(v,k,\lambda)$. For a survey of properties and applications of partial difference sets, see~\cite{ma.partsifset}.

We  first show that difference sets and partial difference sets in $\mathbb{F}_2^n$ correspond to eigenvectors of the convolution matrix $A$ of $\mathbb{F}_2^n$.

\begin{utv} \label{partdifeigen}
Let $A$ be the convolution matrix of the group $\mathbb{F}_2^n$ of order $v = 2^n$. If $D \subseteq \mathbb{F}_2^n$ is a partial difference set with parameters $(v,k,\lambda,\mu)$,  then  there exists an eigenfunction $f$ of $A$  of the form $f = \chi_D  + \beta\textbf{1} + \alpha \chi_{\overline{0}}$   for some $\alpha, \beta \in \mathbb{C}$.
\end{utv}

\begin{proof}
We determine conditions on the parameters $\alpha, \beta \in \mathbb{C}$   under which the function $f = \chi_D  + \beta\textbf{1} + \alpha \chi_{\overline{0}}$  is an eigenfunction   of the convolution matrix $A$. By the definition, $f$ is an eigenfunction  of $A$ corresponding to an eigenvalue $\theta$ if and only if  $A \circ f = \theta (\mathbb{I} \circ f)$. By  Corollary~\ref{convaseigen},  we have $A \circ f  = f*f$.

From the definition of convolution,  for every  set $X\subseteq \mathbb{F}_2^n$  it holds  $\textbf{1}*\chi_X =|X|  \textbf{1}$    and for every  function $g$  we have  $\chi_{\overline{0}}*g=g$. Using these identities together with the commutativity of convolution, we  expand  $f*f$ as follows:
\begin{gather*}
f*f =  \chi_D * \chi_D  +  \beta^2  (\textbf{1} * \textbf{1}) + \alpha^2    (\chi_{ \overline{0} }  * \chi_{\overline{0}})  + 2 \beta  ( \textbf{1}  * \chi_D) + 2 \alpha  (\chi_D * \chi_{\overline{0}}) + 2 \alpha \beta (\textbf{1}   * \chi_{\overline{0}})  = \\
  \chi_D * \chi_D  + 2 \alpha \chi_D   +  \beta (   v \beta  + 2  k  + 2 \alpha ) \textbf{1} + \alpha^2   \chi_{\overline{0}}.
\end{gather*}

The definition of a partial difference set $D$ with parameters $(v,k,\lambda, \mu)$ in $\mathbb{F}_2^n$ is entrywise equivalent to 
$$\chi_D*\chi_D=(\lambda-\mu)  \chi_D+ \mu  \textbf{1}+(k-\mu)  \chi_{\overline{0}}.$$

So, the equation  $A \circ f = \theta (\mathbb{I} \circ f)$ is equivalent to
$$    ( 2 \alpha  + \lambda - \mu) \chi_D   + ( \mu +\beta (   v \beta  + 2  k  + 2 \alpha ))  \textbf{1} + ( \alpha^2 +  k - \mu)  \chi_{ \overline{0} }  = \theta (\mathbb{I} \circ f). $$

We evaluate this vector equation coordinatewise considering the three cases $x \in D$, $x  \not\in D$ but $ x \neq \overline{0} $, and $x = \overline{0}$. Then  we get following system:
$$
\begin{array}{rl}
x \in D :  &    2 \alpha \beta  +   v \beta^2  +  2 \alpha  +   2  k \beta  + \lambda   = \theta  (1 + \beta)^2;  \\
x  \not\in D, x \neq \overline{0} : &       2 \alpha \beta +  v \beta^2  +  2  k \beta    + \mu  =   \theta \beta^2; \\
x  = \overline{0} : &   \alpha^2 +  2 \alpha \beta +  v \beta^2  +  2  k \beta     +  k   =  \theta (\alpha + \beta)^2. \\
\end{array}
$$

Subtraction of the second equation from the first and the third ones gives us

\begin{equation} \label{eigensys}
\left\{
\begin{array}{l}
 2 \alpha   + \lambda  - \mu  = \theta  (1 + 2 \beta);  \\
   2 \alpha \beta +  v \beta^2  +  2  k \beta    + \mu  =   \theta \beta^2; \\
  \alpha^2     +  k  - \mu  =  \theta (\alpha^2 +  2\alpha \beta). \\
\end{array}
\right.
\end{equation} 

Next, subtract the first equation multiplied by $\beta$ from the second one and obtain:

$$ \left\{
\begin{array}{l}
 2 \alpha   + \lambda  - \mu  = \theta  (1 + 2 \beta);  \\
    v \beta^2  +  (2  k - \lambda + \mu) \beta    + \mu    =   - \theta   \beta (\beta + 1); \\
  \alpha^2     +  k  - \mu  =  \theta (\alpha^2 +  2\alpha \beta). \\
\end{array}
\right.
$$

So in case $\beta \neq 0$ and $\beta \neq -1$,   the second equation  implies
$$ 
\theta =  - \frac{v \beta^2  +  (2  k - \lambda + \mu) \beta    + \mu }{\beta(\beta+1)}.
$$

From the first equation of the system we express $\alpha$:
$$ \alpha = \frac{1}{2} \left( \mu - \lambda  - \frac{v \beta^2  +  (2  k - \lambda + \mu) \beta    + \mu }{\beta(\beta+1)}  (1 + 2 \beta)  \right).$$

Substitution of the obtained expressions for $\alpha$ and $\theta$ into the third equation of the last system gives us a rational equation of the form $\frac{P(\beta)}{\beta (\beta +1)} = 0$, where $P(\beta)$ is a polynomial of degree $8$. Solving $P(\beta) = 0$, excluding  singular values $\beta = 0$ and $\beta = -1$, and the  substituting the resulting roots into the expressions for $\alpha$ and $\theta$, we obtain the eigenvalues and eigenfunctions of  the matrix $A$  of the required forms for every partial difference set $D$.

Let us consider the remaining cases in which  partial difference sets give additional eigenfunctions and eigenvalues of special forms.   Suppose that there is an eigenfunction $f$ with $\beta = 0$. Then the second equation of (\ref{eigensys}) implies  $\mu = 0$. Consequently, in this case  $D \cup \{ 0\}$ is a subgroup of $\mathbb{F}_2^n$  and  therefore $\lambda = k -1$. From the first equation we obtain $\theta  = 2\alpha + k -1$. Substituting  this expression to the third equation of $(\ref{eigensys})$  gives  $ 2 \alpha^3  + 2  \alpha k   - k  = 0 $ that allows us to  find $\alpha$  and  $\theta$. 

Assume now that  there exists an eigenfunction $f$ of $A$ with  $\beta = -1$. From the first equation of (\ref{eigensys}) we obtain $\alpha = \frac{1}{2} (\mu - \lambda - \theta)$.  From the second equation  we derive that  $   \lambda  +  v   =    2  k$. It is equivalent to that $G \setminus D$ is a proper subgroup of $\mathbb{F}_2^n$ and therefore  $\mu = k$.     Substituting these expressions into the third equation gives   $ \alpha   ( (1 - \theta) \alpha + 2  \theta) =  0$. It implies either $\alpha = 0$ and $\theta = \mu - \lambda$ or  $2((1 - \theta) \alpha + 2 \theta)  =   \theta^2  + \theta(\lambda - \mu +3)  +  \mu - \lambda = 0$  that allows us to  find $\theta$  and  $\alpha$. 
\end{proof}

\begin{zam}
The statement and the  proof of Proposition~\ref{partdifeigen} remain true for regular partial difference sets $D$ in a general abelian group $G$ satisfying $D = -D$. 
\end{zam}

\textbf{Example 1.} Let  $D$ be a difference set in  $\mathbb{F}_2^n$ with the McFarland parameters $(2^n, 2^{n-1} - 2^{n/2 -1}, 2^{n-2}-
2^{n/2 -1})$, where $n$ is even. Then the system (\ref{eigensys}) has the form
$$
\left\{
\begin{array}{l}
 2 \alpha   = \theta  (1 + 2 \beta);  \\
   2 \alpha \beta +  2^n \beta^2  +  ( 2^{n} - 2^{n/2}) \beta    + 2^{n-2}-
2^{n/2 -1}  =   \theta \beta^2; \\
  \alpha^2     + 2^{n-1}    - 2^{n-2}   =  \theta (\alpha^2 +  2\alpha \beta). \\
\end{array}
\right.
$$

Substituting  $\alpha = \theta (\beta + 1/2)$ from the first equation into the second and third ones gives
$$
\left\{
\begin{array}{l}
( \theta  + 2^n) \beta^2   +  ( 2^{n} - 2^{n/2} + \theta) \beta    + 2^{n-2}-  2^{n/2 -1}  = 0; \\
(\theta^3  + \theta^2 )\beta^2 +  \theta^3  \beta +    (\theta^3  -  \theta^2)/4   - 2^{n-1}    + 2^{n-2}   = 0 . \\
\end{array}
\right.
$$

It can be checked that if $\theta = 2^{n/2}$, then the second equation of the above system coincides with the first one multiplied by $2^{n/2}$.  The values of $\beta$ can be found from the following quadratic equation
$$
(2^n + 2^{n/2}) \beta^2   +  2^{n}  \beta    + 2^{n-2}-  2^{n/2 -1}  =   0 .\\
$$ 

\medskip

Note that the eigenvalue $\theta = 2^{n/2}$ for difference sets with McFarland parameters considered in this example indeed belongs to the spectrum of the convolution matrix $A_{\mathbb{F}_2^n}$ (see Theorem~\ref{F2nvalues}). Similarly, eigenvalues obtained via Proposition~\ref{partdifeigen}  for any other partial difference sets with known parameters must be eigenvalues of convolution matrices $A_{\mathbb{F}_2^n}$ for groups $\mathbb{F}_2^n$. 

Another main result of this section is that the indicator function of   a partial difference set $D \subseteq \mathbb{F}_2^n$  is a  perfect coloring of the subspace hypergraph $\mathcal{H}_n$.

\begin{theorem}  \label{difsetperfcolor}
 The set  $D \subseteq \mathbb{F}_2^n$, $\overline{0} \not\in D$, is a partial difference set  with parameters $(2^n,k,\lambda,\mu)$  if and only if  $\chi_D$ is a perfect coloring of the subspace hypergraph $\mathcal{H}_n$  with the $3$-dimensional  parameter matrix $S$  such that
$$
\begin{array}{ll}
s_{0,0, 0} =  2^{n-1}  - k + \mu/2 -1;  &
s_{1,0, 0} = 2^{n-1}   - k + \lambda/2  ;  \\
s_{0,1, 0} = s_{0,0,1}  = \frac{1}{2} (k - \mu)  ;  &    s_{1,0,1} = s_{1,1,0} =   \frac{1}{2} (k - \lambda -1); \\
s_{0,1,1} = \mu / 2;  &
 s_{1,1,1} = \lambda /2.
\end{array}
$$
\end{theorem}

\begin{proof}
By definition, the indicator function $\chi_D$ is a coloring of $\mathcal{H}_n$ into colors $0$ and $1$. 
To prove that $\chi_D$ is a perfect coloring, we compute explicitly the numbers $t_{i,j}$, $i \in \{0,1\}$, $j \in \{ 0, 1, 2, 3 \}$, where   $t_{i,j}$ denotes the numbers of hyperedges incident to a vertex $x$ of color $i$ and  containing $j$ vertices from the set  $D$. If  these quantities  are independent of the choice of $x$, then, by the definition of a perfect coloring and Theorem~\ref{adjparam}, $\chi_D$ is a perfect coloring of $\mathcal{H}_n$ with the parameter matrix $S$ whose entries are
$$
\begin{array}{ll}
s_{0,0, 0} = t_{0,0};  &
s_{1,0, 0} = t_{1,1};  \\
s_{0,1, 0} = s_{0,0,1}  = t_{0,1} /2;  &    s_{1,0,1} = s_{1,1,0} =  t_{1,2} /2; \\
s_{0,1,1} =  t_{0,2};  &
 s_{1,1,1} = t_{1,3}.
\end{array}
$$

From the definition of a partial difference set, for every $x \in D$ there are $\lambda / 2$ hyperedges $\{ x, y, z \}$ of $\mathcal{H}_n$ such that $y, z \in D$. Indeed, in this case the number of unordered triples $\{ x,y, z \}$ such that  $x + y +z = \overline{0}$ is a half of the number of ordered pairs of $y,z \in D$ such that $x = y - z$. Similarly,  there  are $\mu / 2$ hyperedges $\{ x, y, z \}$ of $\mathcal{H}_n$ for which  $y, z \in  \mathbb{F}_2^n \setminus (D \cup \{ \overline{0}\})$.  It implies that  $t_{1,3} = \lambda/2$ and $t_{0,2} = \mu/2$ and they do not depend on the choice of the vertex $x$ but only on its color. 

On the other hand, by the definition of a partial difference set, for every $x \in D$ there are exactly $\lambda$ other $y \in D$ such that the remaining vertex $z = x + y$  from a hyperedge of $\mathcal{H}_n$ also belongs to the set $D$. Therefore, the remaining $ k - \lambda -  1 $ vertices $y \in D$ give hyperedges $\{ x,y,z \}$, where  the vertex $z$ is not from $D$. It means that $t_{1,2} = k - \lambda - 1$ and it does not depend on $x$. 

Since  $\mathcal{H}_n$ is a regular hypergraph of degree $2^{n-1} - 1$, for every vertex $x \in D$ we have
$$2^{n-1} -1 = t_{1,1} + t_{1,2} + t_{1,3}.$$
So
$$t_{1,1} = 2^{n-1} -1 - t_{1,2} - t_{1,3} =  2^{n-1}   - k + \lambda/2 $$
and $t_{1,1}$ also does not depend on $x$.

Similarly, by the definition of the partial difference set, for every $x \not\in D$, $x \neq \overline{0}$, there are exactly $\mu$ vertices $y \in D$ such that the remaining vertex $z$  from the hyperedge $\{ x,y,z \}$ of $\mathcal{H}_n$ also belongs to the set $D$. Therefore, the remaining $ k - \mu  $ vertices $y \in D$ give hyperedges $ \{ x,y,z \}$, where the vertex $z$ is not from $D$. It means that $t_{0,1} = k - \mu $ and it does not depend on $x$. 
 
 Again, since  $\mathcal{H}_n$ is a regular hypergraph of degree $2^{n-1} -1$, for every vertex $x \not\in D$, $x \neq \overline{0}$,  we have
$$ 2^{n-1} -1 = t_{0,0} +  t_{0,1} + t_{0,2} .$$
Therefore,
$$t_{0,0} =  2^{n-1} -1 - t_{0,1} - t_{0,2} = 2^{n-1}  - k + \mu/2 -1$$
and it does not depend on $x$. 

Thus we find all parameters $t_{i,j}$, $i \in \{0,1\}$, $j \in \{ 0, 1, 2, 3 \}$ and show that they are independent of the choice of the initial vertex $x$.
\end{proof}

\textbf{Example 2.} Suppose that $D$ is a difference set in $\mathbb{F}_2^n$  with the McFarland parameters $(2^n, 2^{n-1} - 2^{n/2 -1}, 2^{n-2}- 2^{n/2 -1})$. Then $\chi_D$ is a perfect coloring of the subspace hypergraph $\mathcal{H}_n$ with the parameter matrix $S$ such that
$$
 \begin{array}{ll}
s_{0,0,0} = 2^{n-3} + 2^{n/2 -2} -1;  & \quad s_{1,0,0} = 2^{n-3} + 2^{n/2-2} ;  \\
s_{0,0,1} = s_{0,1,0} = 2^{n -3};  &   \quad  s_{1,1,0} = s_{1,0,1} = 2^{n-3} - 1/2; \\
s_{0,1,1} = 2^{n-3} - 2^{n/2-2};  &   \quad s_{1,1,1} = 2^{n-3} - 2^{n/2-2}.
\end{array}
$$

\subsection{Bent and plateaued functions}

Let $f: \mathbb{F}_2^n \rightarrow \mathbb{F}_2$ be a Boolean function. The \textit{support} of a Boolean function $f$ is the set of all $x$  such that $f(x) =1$, and the  \textit{weight} of $f$ is the size of support.  A function $f$ is a \textit{bent function} if all its Fourier coefficients $\widehat{f}(x)$ are equal to $\pm 1$.  A Boolean function $f$ is an \textit{$s$-plateaued function} if  all its Fourier coefficients   are equal to $\pm 2^{\frac{s}{2}}$ or $0$,  for some nonnegative integer $s$. In particular, $0$-plateaued  functions coincide with  bent functions. It is known that every bent function on $n$ variables has weight $2^{n-1} \pm 2^{n/2-1}$.  The cryptographic properties of bent and plateaued functions were extensively studied in~\cite{carle.boolforcrypt,mesnager.bentfunc}. 

There is an equivalent definition of bent functions and plateaued functions in terms of convolution.

\begin{lemma}\cite[Proposition 97]{carle.boolforcrypt}  \label{platdef}
A Boolean  function $f: \mathbb{F}_2^n \rightarrow \mathbb{F}_2$  is $s$-plateaued if and only if 
${(-1)^f}*{(-1)^f}*{(-1)^f}=2^{n+s}{(-1)^f}$.
\end{lemma}

This property implies that   if $f$ is  a plateaued function, then the function $(-1)^f$  is a eigenfunction of the $2$-times convolution matrix.

\begin{theorem} \label{plateigen}
Let $A^{(2)}$ be the  $2$-times convolution matrix of $\mathbb{F}_2^n$ and  let $f: \mathbb{F}_2^n \rightarrow \mathbb{F}_2$  be a Boolean function. Then $f$ is an $s$-plateaued function  if and only if  $(-1)^f$   is an eigenfunction of the matrix $A^{(2)}$ with eigenvalue $2^{n+s}$. 
\end{theorem}

\begin{proof}
The statement follows from Proposition~\ref{convaseigen}, Lemma~\ref{platdef}, and the fact that for the $4$-dimensional identity matrix $\mathbb{I}$ it holds $\mathbb{I} \circ (-1)^f = (-1)^f $. 
\end{proof}

Next, using the results of Section~\ref{relationsec}, we show that plateaued functions are perfect colorings of the subspace hypergraph $\mathcal{D}_n$.

\begin{theorem}
 If  a Boolean function $f: \mathbb{F}_2^n \rightarrow \mathbb{F}_2$  is $s$-plateaued, then the  $2$-coloring $f$ of the vertices of the subspace hypergraph  $\mathcal{D}_{n}$ is  perfect. 
\end{theorem}

\begin{proof}
Note that  functions $f$ and $(-1)^f$ define the same coloring of the hypergraph $\mathcal{D}_n$.  By Theorem~\ref{plateigen}, the function $(-1)^f$ is an eigenfunction for the $2$-times convolution matrix $A^{(2)}$ with eigenvalue $2^{n+s}$.

Let $M_D$ be the adjacency matrix of the hypergraph $\mathcal{D}_n$. 
 By Proposition~\ref{2convhyper4} and the  equality  $\mathbb{I} \circ (-1)^f = (-1)^f $, for all $x \in \mathbb{F}_2^n$ we have
\begin{gather*} 2^{n+s}  (\mathbb{I} \circ (-1)^f) (x) =  (A^{(2)} \circ (-1)^f) (x) =  \\
6 (M_{D} \circ (-1)^f) (x)  + 3  (2^n - 1)  \cdot (-1)^f (x) +  (\mathbb{I} \circ (-1)^f) (x). 
\end{gather*}
Consequently,
$$  M_{D} \circ (-1)^f = \frac{1}{6} ( 2^{n+s} - 3 \cdot 2^{n} +2) (\mathbb{I} \circ (-1)^f)$$
which means that $(-1)^f$ is a two-valued eigenfunction of the hypergraph $\mathcal{D}_n$. 
Since the subspace hypergraph $\mathcal{D}_{n}$ is totally regular, Theorem~\ref{eigencolor} implies that $(-1)^f$ is a perfect $2$-coloring.
\end{proof}

It is known  (see, e.g.,~\cite{carle.boolforcrypt}) that the support  of a Boolean bent function is a  difference set with the McFarland parameters in $\mathbb{F}_2^n$. Our previous results imply that they are perfect $2$-colorings of the subspace hypergraph $\mathcal{H}_n$.

\begin{utv}
 A Boolean function $f : \mathbb{F}_2^n \rightarrow \mathbb{F}_2$  is a bent function of weight $2^{n-1} - 2^{n/2 -1}$ if and only if it is a perfect $2$-coloring  of the subspace hypergraph $\mathcal{H}_n$ with the parameter matrix $S$ such that
$$
 \begin{array}{ll}
s_{0,0,0} = 2^{n-3} + 2^{n/2 -2} -1;  & s_{1,0,0} = 2^{n-3} + 2^{n/2-2} ;  \\
s_{0,0,1} = s_{0,1,0} = 2^{n -3};  &    s_{1,1,0} = s_{1,0,1} = 2^{n-3} - 1/2; \\
s_{0,1,1} = 2^{n-3} - 2^{n/2-2};  &  s_{1,1,1} = 2^{n-3} - 2^{n/2-2}.
\end{array}
$$
\end{utv}

\begin{proof}
The statement follows  from Theorem~\ref{difsetperfcolor} and Example 2.
\end{proof}

\subsection{Spreads}

 A \textit{$k$-spread} is a  partition  $\{ P_1, \ldots, P_m\}$, where $m = \frac{2^n -1}{2^k - 1}$, of the set $\mathbb{F}^n_2\setminus\{ \overline{0}\}$ into $k$-dimensional linear subspaces $P_j$. In this paper we, consider only   $n/2$-spreads of $\mathbb{F}^n_2$ for even $n$. It is well known~\cite[Proposition 80]{carle.boolforcrypt} that a union of  $2^{n/2 - 1}$ subspaces of an $n/2$-spread  of $\mathbb{F}^n_2$  is the support of a Boolean bent function.  For  constructions of  $n/2$-spreads,  see~\cite{Bu.spacepart,Lavrauw.spreadsite}. 

 The main result of this section is that  $n/2$-spreads of $\mathbb{F}^n_2$ are perfect colorings of the subspace hypergraph $\mathcal{H}_n$. To prove this result, we need the following auxiliary statement.
 
\begin{utv} \label{hyperspaceutv}
 Let $P$ be a proper subset of  $\mathbb{F}^n_2$ with $\overline{0}\in P$.  Then there exists an $(n-1)$-dimensional subspace  (hyperspace) $\Gamma \subset \mathbb{F}^n_2$ such that $|\Gamma\cap P| > | (\mathbb{F}^n_2\setminus\Gamma)\cap P|$.
\end{utv}
 
 \begin{proof}
  Each hyperspace $\Gamma$  in   $\mathbb{F}^n_2$ is uniquely  determined by a nonzero orthogonal vector $z$. Denote  $\Gamma_z= \{x\in\mathbb{F}^n_2: (x,z)=0\}$. By the definition of the Fourier transform,   $\widehat{{\chi}_P}(z)= \frac{1}{2^{n/2}} ( |\Gamma_z\cap P|-|(\mathbb{F}^n_2\setminus\Gamma_z)\cap P| )$ for $z \neq \overline{0}$ and $\widehat{{\chi}_P}(\overline{0})=\frac{1}{2^{n/2}} |P|$.  It is also not hard to see that  $\sum\limits_{z \in\mathbb{F}^n_2}\widehat{{\chi}_P}(z)=2^{n/2} \cdot {\chi}_P(\overline{0})=2^{n/2}$.  If $|P| \neq 2^n$, then there exists   $z\neq \overline{0}$ such that $\widehat{{\chi}_P}(z)>0$, which implies $|\Gamma_z\cap P| > | (\mathbb{F}^n_2\setminus\Gamma_z)\cap P|$. 
 \end{proof}

\begin{theorem} \label{spreadascolor}
Let $n$ be even.  A partition   $P = \{ P_1, \ldots, P_m\}$, $m = 2^{n/2} +1$, of  $\mathbb{F}^n_2 \setminus \{ \overline{0}\}$ is an $n/2$-spread  if and only if it defines a  perfect coloring of  the subspace hypergraph $\mathcal{H}_n$ with the parameter matrix $S$ whose entries are:
\begin{itemize}
\item $s_{i,j,k} = 1/2$ for all distinct $i,j,k \in \{ 1, \ldots, m\}$;  
\item $s_{i,i,j} = s_{i,j,i} = s_{j,i,i} = 0$   for all distinct $i,j \in \{ 1, \ldots, m\}$; 
\item $s_{i,i,i} =  2^{n/2 - 1} -1$  for  all $i \in \{ 1, \ldots, m\}$. 
\end{itemize}
\end{theorem}

\begin{proof}
$\Rightarrow$ Let $f$ be a coloring of  $\mathcal{H}_n$ such that $f (x) = i$ if and only if $x \in P_i$.  We show that $f$ is a perfect coloring by explicitly computing the numbers $v_{i,j,k}$ of hyperedges of $\mathcal{H}_n$ with the color range $\{ i,j,k \}$  incident to a vertex $x$ of color $i$. This reasoning is essentially equivalent to computing the convolutions $F_i * F_j$ of the indicator functions $F_i$ of the color classes $P_i$ (see Lemma~\ref{colorasconv}).

The value $v_{i,i,i}$ equals the number of $2$-dimensional linear subspaces of the $n/2$-dimensional linear subspace that contain a given vertex $x$. Hence, $v_{i,i,i} = 2^{n/2 - 1} -1$ and, by Theorem~\ref{adjparam}, $s_{i,i,i} = 2^{n/2 - 1} -1$.

Note that if two vertices from a hyperedge of $\mathcal{H}_n$ belong to one subspace $P_j$, then the third one also belongs to $P_i$. Therefore, there are no hyperedges of color ranges  $\{ i,i,j\}$ and $\{ i,j,j \}$ for $i\neq j$, thus   $v_{i,i,j} = v_{i,j,j} = 0$ and consequently $s_{i,i,j} = s_{i,j,j} = 0$. 

Let  us find $v_{i,j,k}$ for pairwise distinct $i,j,k$. We start with proving that $v_{i,j,k} \leq 1$. Suppose that there exist non-equal $y_1, y_2 \in P_j$ and $z_1, z_2 \in P_k$ such that $x + y_1 + z_1 = \overline{0}$ and $x + y_2 + z_2 = \overline{0}$. Then   $y_1  - y_2 = z_2-z_1 \neq \overline{0}$ which implies that $P_j \cap P_k \neq \{\overline{0} \}$, a contradiction. Consequently, for every $j \neq i$ the set $x + P_j$ intersects any other subspace $P_k$ in at most one element.  

Since  $\{ P_1, \ldots, P_m\}$ is an $n/2$-spread, the direct sum $P_i \oplus P_j$ coincides with $\mathbb{F}^n_2$  for all  $i \neq j$.   Thus, for every $y \in P_j$ there is exactly one $k\neq i,j$ and exactly one $z \in P_k$ such that $x + y +z =\overline{0}$. Therefore, $v_{i,j,k} = 1$  and, by Theorem~\ref{adjparam}, $s_{i,j,k} = 1/2$  if  $i,j,k$ are pairwise distinct. 

$\Leftarrow$  Let us prove that every perfect coloring $f$ of the subspace hypergraph $\mathcal{H}_n$ with the parameter matrix $S$ is an $n/2$-spread. By the symmetry of entries of the parameter matrix $S$, the coloring $f$  defines a  partition  $\{ P_1, \ldots, P_m\}$  of the vertex set of $\mathcal{H}_n$ into $m = 2^{n/2} + 1$ subsets of equal size $2^{n/2}-1$. Moreover,  each set $P_i \cup \{ \overline{0}\}$ contains exactly  $ \frac{1}{3} (2^{n/2}-1) v_{i,i,i} = \frac{1}{3} (2^{n/2}-1) (2^{n/2 - 1} -1)$ $2$-dimensional subspaces of $\mathbb{F}_2^n$.

To complete the proof of the theorem it suffices to  show that every subset   $P\subseteq \mathbb{F}^n_2$ with $\overline{0}\in P$  contains at most $\frac{1}{3}(|P|-1)(\frac{|P|}{2}-1)$  $2$-dimensional subspaces with equality if and only if $P$ is linear subspace. 

We prove this statement by induction on $n$.  The base of induction $n = 2$ is trivial. The statement is also clear if $P = \mathbb{F}^n_2$. Hence, it remains to consider the statement for proper subsets $P$ of $\mathbb{F}^n_2$.

By Proposition~\ref{hyperspaceutv}, there exists a hyperspace   $\Gamma_z= \{x\in\mathbb{F}^n_2: (x,z)=0\}$  such that $|\Gamma_z \cap P| > | (\mathbb{F}^n_2\setminus\Gamma_z)\cap P|$. Denote $p_1 = |\Gamma_z\cap P|$  and $p_2 = | (\mathbb{F}^n_2\setminus\Gamma_z)\cap P|$, $|P| = p_1 + p_2$. If $p_1 = |P|$, then $P \subseteq \Gamma_z$, and identifying $\Gamma_z$ and $\mathbb{F}_2^{n-1}$, the result follows from the inductive assumption.

Suppose now that $ 0 < p_2 < p_1 < |P|$.  We estimate the number of $2$-dimensional subspaces contained in $P$.  In any  $2$-dimensional  subspace  $\{ \overline{0}, x_1, x_2, x_3\}$, either none or exactly two of the nonzero elements lie outside $\Gamma_z$. Indeed, if $x_i$ and $x_j$ are not from $\Gamma_z$, i.e., $(x_i, z) = (x_j, z) = 1$, then $x_i +x_j$ belongs to $\Gamma_z$ because  $(x_i + x_j, z) =0$. In particular, every  pair of elements  from $\mathbb{F}^n_2\setminus\Gamma_z$ is contained in at most one $2$-dimensional subspace in  $P$.  

Consequently,  the number of $2$-dimensional subspaces in $P$ does not exceed the number of such subspaces in $P \cap \Gamma_z$ plus  the number of unordered pairs of elements from $(\mathbb{F}^n_2\setminus\Gamma_z)\cap P$. By the inductive assumption, this number is at most
$$\frac{1}{3}(p_1-1)(\frac{p_1}{2}-1)+\frac{1}{2}p_2(p_2-1).$$
Since this quantity is strictly less than
$\frac{1}{3}(p_1+p_2-1)(\frac{p_1+p_2}{2}-1) = \frac{1}{3}(|P|-1)(\frac{|P|}{2}-1)  $ whenever $p_1>p_2>0$, the proof is complete. 
\end{proof}

\subsection{Bent partitions}

A partition $P= \{U,P_1,\ldots,P_K\}$ of $\mathbb{F}^n_2$ is called  a \textit{normal bent partition of depth $K$}, if every function $f$ satisfying the following properties is bent:
\begin{enumerate}
\item  For each  $c\in  \{ 0,1 \}$, the preimage $f^{-1} (c)$ consists of  exactly  $K/2$   sets among  $P_1,\dots,P_K$;
\item The function $f$ is constant on the set $U$.
\end{enumerate}

The following property of  normal bent partitions was proved in~\cite{AnKaMei.bentpart}.

\begin{utv} \cite[Lemma 1]{AnKaMei.bentpart} 
Let $P= \{U,P_1,\ldots,P_K\}$  be    a normal bent partition of $\mathbb{F}^n_2$ of depth $K$. Then $K = 2^k$ for some $k$,  the set $U$ is an  $n/2$-dimensional affine subspace of $\mathbb{F}^n_2$, and  all sets $P_i$, $1\leq i\leq 2^k$, have the same size: $|P_i | = \frac{2^{n} - 2^{n/2}}{2^k}$.
\end{utv}

 It is known~\cite{carle.boolforcrypt} that an affine transformation  of  the argument  of a bent function is again a bent function. Therefore, every affine transformation of $\mathbb{F}^n_2$ preserves the property of a partition being a normal bent partition. Hence, in what follows we may assume   that  $U$ is a linear subspace.  
 
As noted in~\cite{AnKaMei.bentpart}, a  normal  bent partition $\{U,P_1,\ldots, P_{2^{n/2}}\}$  of $\mathbb{F}^n_2$ of depth $2^{n/2}$ is equivalent to an $n/2$-spread. However,  there exist normal bent partitions of depth $2^{n/2}$ that are different from spreads~\cite{MeidlPir.bentforspred}.  If a  partition $R =  \{U,R_1,\ldots, R_{2^{k-m}}\}$ is obtained  from a normal bent partition $P = \{U,P_1,\ldots, P_{2^{k}}\}$ by  uniting arbitrary groups of $2^m$ subsets $P_i$,  where $m < k$,  then $R$ is also a normal bent partition. Meanwhile,  not every normal bent partition  can be obtained in this way.  

The main aim of this section is to show that certain bent partitions are equitable partitions of the subspace hypergraph $\mathcal{H}_n$. For this purpose we need  additional properties of strong bent partitions.

\begin{utv}\cite[Lemma 6]{AnKaMei.bentpart}  \label{bentpartForier}
Let $\{U,P_1,\dots, P_{2^k}\}$ be a normal bent partition of $\mathbb{F}^n_2$. Then $\widehat{\chi_{P_i}} (z)$ takes the following values:
\begin{itemize}
\item If $z\in U^\perp\setminus \{\overline{0}\}$, then $\widehat{{\chi}_{P_i}}(z)=-1/2^k$  for all $i=1,\ldots,2^k$.
\item   $\widehat{{\chi}_{P_i}}(\overline{0})=\frac{2^{n/2} - 1}{2^k}$  for all $i=1,\ldots,2^k$. 
\item For any other $z\not\in U^\perp$, one of two possibilities occurs:
\begin{enumerate}
\item  There exists an unique index $j$ such that  $\widehat{{\chi}_{P_i}}(z)=-1/2^k$ for all $i$,  except for  $j$, where $\widehat{{\chi}_{P_j}}(z)=1-1/2^k$.  We will say that such $z$  belongs to a subset $ R_j$.
\item   There exists an unique index $j$ such that  $\widehat{{\chi}_{P_i}}(z)=1/2^k$ for all $i$,  except for $j$, where $\widehat{{\chi}_{P_j}}(z)=1/2^k-1$. 
\end{enumerate}
\end{itemize}   
\end{utv}

If, for a normal bent partition $P = \{U,P_1,\dots, P_{2^k}\}$,   the second possibility  in the last case of Proposition~\ref{bentpartForier} never occurs for any $z\not\in U^\perp$, then   $P$ is  called  a  \textit{strong bent partition}.  Some constructions of   strong bent partitions can be found in~\cite{AnKaMei.bentpart}. Moreover, in that paper a correspondence between strong bent partitions  and associative schemes was established. It was also shown that if $\{U,P_1,\dots, P_k\}$ is a strong bent partition of $\mathbb{F}^n_2$, then  for each   $i \in \{ 1,\ldots,k\}$  the set  $P_i$ is a partial difference set.

To prove  that  a strong bent partition yields  a perfect coloring of $\mathcal{H}_n$, we need the following corollary of Proposition~\ref{bentpartForier}.

\begin{utv} \label{dualpart}
If $\{U,P_1,\dots, P_{2^k}\}$ is a strong bent partition of $\mathbb{F}^n_2$, then there exist sets $R_1, \ldots, R_{2^k}$ such that  $\widehat{{\chi}_{P_j}}={\chi}_{R_j}-\frac{1}{2^k}{\bf 1} + \frac{2^{n/2}}{2^{k}} \chi_{\overline{0}}$ and  $\{U^\perp,R_1,\dots, R_{2^k}\}$ is a strong bent partition of $\mathbb{F}^n_2$.
\end{utv}

\begin{proof}
Since $\{U,P_1,\dots, P_{2^k}\}$ is a strong bent partition,   Proposition~\ref{bentpartForier}  implies the existence of sets  $R_1, \ldots, R_{2^k}$ such that  $\widehat{{\chi}_{P_j}}={\chi}_{R_j}-\frac{1}{2^k}{\bf 1} + \frac{2^{n/2}}{2^{k}} \chi_{\overline{0}}$. 
Using this equality together with Proposition~\ref{fourierutv}, we find that all sets  $R_j$ have the same size:
$$|R_j| = 2^{n/2}  \cdot \widehat{\chi_{R_j}} (\overline{0})   =  2^{n/2} \left( \chi_{P_j} (\overline{0}) +\frac{1}{2^k} \widehat{{\bf 1}} (\overline{0})  - \frac{2^{n/2}}{2^{k}}  \widehat{\chi_{\overline{0}}} (\overline{0})   \right) =  \frac{2^n - 2^{n/2}}{2^k}.  $$

It remains to show that $\{U^\perp,R_1,\dots, R_{2^k}\}$  is a strong bent partition. We first prove that  $\{U^\perp,R_1,\dots, R_{2^k}\}$ is a normal bent partition. 

The condition that $\{U,P_1,\dots, P_{2^k}\}$  is a normal bent partition can be reformulated as follows:  for  every partition $I\cup J$ of the index set $\{1,\dots,2^k\}$   with  $|I|=|J|$,  the Boolean function $f$ defined by the equation
$$(-1)^f={\chi}_U+ \sum\limits_{i\in I} {\chi}_{P_i}-\sum\limits_{j\in J}{\chi}_{P_j}$$ is bent. Applying the Fourier transform to the both sides yields
$$\widehat{(-1)^f} =\widehat{{\chi}_U} + \sum\limits_{i\in I} \widehat{{\chi}_{P_i}}-\sum\limits_{j\in J} \widehat{{\chi}_{P_j}} = {\chi}_{U^\perp}+ \sum\limits_{i\in I} {\chi}_{R_i}-\sum\limits_{j\in J}{\chi}_{R_j}.$$

It is well known~\cite{carle.boolforcrypt} that $f$ is a bent if and only if  there exists a Boolean function $f'$ such  that $\widehat{(-1)^f}=(-1)^{f'}$ and $f'$ is also a bent.  Hence, the above equality implies that  $\{U^\perp,R_1,\dots, R_{2^k}\}$ is a normal bent partition.  

The fact that  $\{U^\perp,R_1,\dots, R_{2^k}\}$ is a strong bent partition follows from equalities $ \widehat{\chi_{R_j}}   =     \chi_{P_j} - \frac{1}{2^{k}}  \textbf{1} +\frac{2^{n/2}}{2^k}  \chi_{\overline{0}} $, $j = 1, \ldots, 2^k$, which are obtained by an application of the Fourier transform to the similar relations for $\widehat{\chi_{P_j}}$.
\end{proof}

We call the strong bent partition  $\{U^\perp,R_1,\dots, R_{2^k}\}$ constructed in Proposition~\ref{dualpart}  the \textit{dual bent partition} of $\{U,P_1,\dots, P_{2^k}\}$.

We are now ready to prove the main result of this section.

\begin{theorem} \label{bentpartascolor}
If $\{U,P_1,\ldots, P_{2^k}\}$ is a strong bent partition of $\mathbb{F}^n_2$, then the  coloring of the subspace hypergraph $\mathcal{H}_n$   into $2^k+1$  colors  $\{U \setminus \{ \overline{0}\},P_1,\ldots, P_{2^k}\}$  is a perfect coloring.
\end{theorem}

\begin{proof}
By Lemma~\ref{colorasconv}, the partition $\{U \setminus \{ \overline{0}\},P_1,\ldots, P_{2^k}\}$ of  $\mathbb{F}_2^n$  is a perfect coloring of $\mathcal{H}_n$ if and only if the convolution of any two indicator functions of these sets can be expressed as a linear combination of all color indicator functions.  Thus, it suffices to compute the convolutions ${\chi}_{P_i}*{\chi}_{P_j}$,  $ \chi_{U \setminus \{ \overline{0}\}} * \chi_{P_i}$, and $ \chi_{U \setminus \{ \overline{0}\}}* \chi_{U \setminus \{ \overline{0}\}}$.   Since the vertex set of  $\mathcal{H}_n$  is $\mathbb{F}_2^n \setminus \{ \overline{0}\}$, the values of these convolution  at $\overline{0}$ are irrelevant.   In what follows, we also repeatedly use  that for $x \neq \overline{0}$, $\chi_{\overline{0}} = 0$ and $\textbf{1} (x) = \sum\limits_{i=1}^{2^k} \chi_{P_i} + \chi_{U \setminus \{\overline{0} \}} = \sum\limits_{i=1}^{2^k} \chi_{R_i} + \chi_{U^\perp \setminus \{\overline{0} \}} $.

Consider ${\chi}_{P_i}*{\chi}_{P_j}$ for distinct $i,j \in \{ 1,\dots,2^k \}$.    Using  the properties of the Fourier transform from Proposition~\ref{fourierutv}  and the existence of  the dual strong bent  partition $\{U^\perp,R_1,\dots, R_{2^k}\}$ for which 
\begin{equation} \label{preq}
\widehat{{\chi}_{P_j}}={\chi}_{R_j}-\frac{1}{2^k}{\bf 1} + \frac{2^{n/2}}{2^{k}} \chi_{\overline{0}},
\end{equation}
 by Proposition~\ref{dualpart}, we have
\begin{gather*}
\widehat{{\chi}_{P_i}*{\chi}_{P_j}}  = 2^{n/2} \cdot  \widehat{{\chi}_{P_i}} \cdot \widehat{{\chi}_{P_j}} =  2^{n/2} \cdot \left({\chi}_{R_i}-\frac{1}{2^k}{\bf 1} + \frac{2^{n/2}}{2^{k}} \chi_{\overline{0}}\right)  \left({\chi}_{R_j}-\frac{1}{2^k}{\bf 1} +\frac{2^{n/2}}{2^{k}} \chi_{\overline{0}}\right)  = \\
  - \frac{2^{n/2}}{2^k} ({\chi}_{R_i}  + {\chi}_{R_j}) + \frac{2^{n/2}}{2^{2k}}{\bf 1}   + \frac{2^{3n/2} - 2^{n+1}}{2^{2k}}  \chi_{\overline{0}}.
\end{gather*}

Consequently, by  Proposition~\ref{fourierutv} and equation~(\ref{preq}), we have
\begin{gather*}
{\chi}_{P_i}*{\chi}_{P_j}= \widehat{\widehat{{\chi}_{P_i}*{\chi}_{P_j}}} =   - \frac{2^{n/2}}{2^k} (\widehat{{\chi}_{R_i}}  +\widehat{{\chi}_{R_j}}) + \frac{2^{n/2}}{2^{2k}} \widehat{{\bf 1}}   + \frac{2^{3n/2} - 2^{n+1}}{2^{2k}}   \widehat{\chi_{\overline{0}}} = \\
- \frac{2^{n/2}}{2^k} ({\chi}_{P_i}  +{\chi}_{P_j})   + \frac{2^{n}}{2^{2k}}   ( \textbf{1}  -  \chi_{\overline{0}} ). 
\end{gather*}

If $i = j$, then 
\begin{gather*}
\widehat{{\chi}_{P_i}*{\chi}_{P_i}}  = 2^{n/2} \cdot  \widehat{{\chi}_{P_i}}^2=  2^{n/2} \cdot \left({\chi}_{R_i}-\frac{1}{2^k}{\bf 1} + \frac{2^{n/2}}{2^{k}} \chi_{\overline{0}}\right)^2 = \\
2^{n/2}   \left(1 - \frac{2}{2^k}\right) \chi_{R_i} + \frac{2^{n/2}}{2^{2k}} \textbf{1} + \frac{2^{3n/2} - 2^{n+1} } {2^{2k}} \chi_{\overline{0}}
\end{gather*}
and
\begin{gather*}
{\chi}_{P_i}*{\chi}_{P_i}= \widehat{\widehat{{\chi}_{P_i}*{\chi}_{P_i}}} =  2^{n/2}   \left(1 - \frac{2}{2^k}\right) \widehat{\chi_{R_i}} + \frac{2^{n/2}}{2^{2k}} \widehat{\textbf{1}} + \frac{2^{3n/2} - 2^{n+1} } {2^{2k}} \widehat{\chi_{\overline{0}}} = \\
2^{n/2}   \left(1 - \frac{2}{2^k}\right)  \chi_{P_i}  +
  \left( \frac{2^{n}} {2^{2k}}  - \frac{2^{n/2}}{2^k}  \right)  \textbf{1}  + 
   \left(\frac{2^{n}}{2^k} - \frac{2^{n}}{2^{2k}}\right)  \chi_{\overline{0}} .
\end{gather*}

 Next,  for each $i \in \{ 1,\dots,2^k \}$  consider $\chi_{U \setminus \{ \overline{0}\}} * \chi_{P_i}$.  Since $U$  is an $n/2$-dimensional linear subspace,  Proposition~\ref{fourierutv}  implies that $\widehat{\chi_U} = \chi_{U^\perp}$.  Using the fact that $\chi_U = \chi_{U \setminus \{ \overline{0}\}} + \chi_{\overline{0}}$, Proposition~\ref{fourierutv}, and equation~(\ref{preq}), we get that
\begin{gather*}
\widehat{\chi_{U \setminus \{ \overline{0}\}} * \chi_{P_i}} = 2^{n/2} \cdot  (\widehat{\chi_U}  - \widehat{\chi_{\overline{0}}})  \widehat{\chi_{P_i}} =  2^{n/2} \cdot \left(\chi_{U^\perp}  - \frac{1}{2^{n/2}} \textbf{1} \right)  \left({\chi}_{R_i}-\frac{1}{2^k}{\bf 1} + \frac{2^{n/2}}{2^{k}}  \chi_{\overline{0}}\right) = \\
 -\frac{2^{n/2}}{2^k}  \chi_{U^\perp}  - \chi_{R_i} + \frac{1}{2^k} \textbf{1}  + \frac{2^{n} - 2^{n/2}}{2^{k}}  \chi_{\overline{0}}
\end{gather*}
and 
\begin{gather*}
\chi_{U \setminus \{ \overline{0}\}} * \chi_{P_i} = \widehat{\widehat{\chi_{U \setminus \{ \overline{0}\}} * \chi_{P_i}}} =   -\frac{2^{n/2}}{2^k}  \chi_{U \setminus \{ \overline{0}\}}  - \chi_{P_i}   + \frac{2^{n/2} }{2^{k}}  \textbf{1}  -\frac{2^{n/2}}{2^k}  \chi_{\overline{0}}  .
\end{gather*}

At last, consider the function $\chi_{U \setminus \{ \overline{0}\}} * \chi_{U \setminus \{ \overline{0}\}}$. Then 
$$\widehat{\chi_{U \setminus \{ \overline{0}\}} * \chi_{U \setminus \{ \overline{0}\}}} = 2^{n/2}  \cdot ( \chi_{U^\perp}  - \widehat{\chi_{\overline{0}}})^2  =   (2^{n/2} - 2) \chi_{U^\perp} + \frac{1}{2^{n/2}} \textbf{1} $$
and 
$$\chi_{U \setminus \{ \overline{0}\}} * \chi_{U \setminus \{ \overline{0}\}} = \widehat{\widehat{\chi_{U \setminus \{ \overline{0}\}} * \chi_{U \setminus \{ \overline{0}\}}}} =  (2^{n/2} - 2) \chi_{U \setminus \{ \overline{0}\}} +   (2^{n/2} - 1)  \chi_{\overline{0}}.$$

Consequently,   ${\chi}_{P_i}*{\chi}_{P_j}$,  $ \chi_{U \setminus \{ \overline{0}\}} * \chi_{P_i}$, and $ \chi_{U \setminus \{ \overline{0}\}}* \chi_{U \setminus \{ \overline{0}\}}$ evaluated at points $x \neq \overline{0}$ can be expressed as  linear combinations of the indicator functions  $\chi_{U \setminus \{ \overline{0}\}}$ and $\chi_{P_\ell}$, $\ell =1, \ldots, 2^k$. Thus, by Lemma~\ref{colorasconv},  we have that $\{\chi_{U \setminus \{ \overline{0}\}},P_1,\ldots, P_{2^k}\}$  is a perfect coloring of $\mathcal{H}_n$.

\end{proof}

\section*{Acknowledgements}
The work of Anna Taranenko is supported by the state contract of the Sobolev Institute of Mathematics (project no. FWNF-2026-0011).

\end{document}